\documentclass[12pt, reqno]{amsart}
\usepackage{}
\usepackage{stmaryrd}
\usepackage{amsmath}
\usepackage{amsfonts}
\usepackage{amssymb}
\usepackage[all,cmtip]{xy}           %xypic macro for latex2.09
\usepackage{xspace}
\usepackage{bbding}
\usepackage{txfonts}
\usepackage[shortlabels]{enumitem}
\usepackage{cite}

\usepackage{tikz}
\usepackage{titletoc}

\usepackage{ifpdf}
\ifpdf
 \usepackage[colorlinks,final,backref=page,hyperindex]{hyperref}
\else
 \usepackage[colorlinks,final,backref=page,hyperindex,hypertex]{hyperref}
\fi
\usepackage[active]{srcltx}

\makeatletter

\begin{document}
\newcommand {\emptycomment}[1]{} %to remove paragraphs

\baselineskip=15pt
\newcommand{\nc}{\newcommand}
\newcommand{\delete}[1]{}
\nc{\mfootnote}[1]{\footnote{#1}} % Use this to show footnotes
\nc{\todo}[1]{\tred{To do:} #1}

%\delete{
\nc{\mlabel}[1]{\label{#1}}  % Use this to suppress names
\nc{\mcite}[1]{\cite{#1}}  % Use this to suppress names
\nc{\mref}[1]{\ref{#1}}  % Use this to suppress names
\nc{\meqref}[1]{\eqref{#1}} % Use this to suppress names
\nc{\mbibitem}[1]{\bibitem{#1}} % Use this to show number
%}

\delete{
\nc{\mlabel}[1]{\label{#1}  % Use the next two lines to show names
{\hfill \hspace{1cm}{\bf{{\ }\hfill(#1)}}}}
\nc{\mcite}[1]{\cite{#1}{{\bf{{\ }(#1)}}}}  % Use this lines to show names
\nc{\mref}[1]{\ref{#1}{{\bf{{\ }(#1)}}}}  % Use this lines to show names
\nc{\meqref}[1]{\eqref{#1}{{\bf{{\ }(#1)}}}} % Use this lines to show names
\nc{\mbibitem}[1]{\bibitem[\bf #1]{#1}} % Use this to show name
}

\newcommand {\comment}[1]{{\marginpar{*}\scriptsize\textbf{Comments:} #1}}
\nc{\mrm}[1]{{\rm #1}}
\nc{\id}{\mrm{id}}  \nc{\Id}{\mrm{Id}}
%%%%%%%%%%%%%%%%%%%%%%%%

\def\a{\alpha}
\def\b{\beta}
\def\bd{\boxdot}
\def\bbf{\bar{f}}
\def\bF{\bar{F}}
\def\bbF{\bar{\bar{F}}}
\def\bbbf{\bar{\bar{f}}}
\def\bg{\bar{g}}
\def\bG{\bar{G}}
\def\bbG{\bar{\bar{G}}}
\def\bbg{\bar{\bar{g}}}
\def\bT{\bar{T}}
\def\bt{\bar{t}}
\def\bbT{\bar{\bar{T}}}
\def\bbt{\bar{\bar{t}}}
\def\bR{\bar{R}}
\def\br{\bar{r}}
\def\bbR{\bar{\bar{R}}}
\def\bbr{\bar{\bar{r}}}
\def\bu{\bar{u}}
\def\bU{\bar{U}}
\def\bbU{\bar{\bar{U}}}
\def\bbu{\bar{\bar{u}}}
\def\bw{\bar{w}}
\def\bW{\bar{W}}
\def\bbW{\bar{\bar{W}}}
\def\bbw{\bar{\bar{w}}}
\def\btl{\blacktriangleright}
\def\btr{\blacktriangleleft}
\def\ci{\circ}
\def\d{\delta}
\def\dd{\diamondsuit}
\def\D{\Delta}
\def\G{\Gamma}
\def\g{\gamma}
\def\k{\kappa}
\def\l{\lambda}
\def\lr{\longrightarrow}
\def\o{\otimes}
\def\om{\omega}
\def\p{\psi}
\def\r{\rho}
\def\ra{\rightarrow}
\def\rh{\rightharpoonup}
\def\lh{\leftharpoonup}
\def\s{\sigma}
\def\st{\star}
\def\ti{\times}
\def\tl{\triangleright}
\def\tr{\triangleleft}
\def\v{\varepsilon}
\def\vp{\varphi}
\def\vth{\vartheta}

%%%%%%%%%%%%%%%%%%%%%%%% Statements
\newtheorem{thm}{Theorem}[section]
\newtheorem{lem}[thm]{Lemma}
\newtheorem{cor}[thm]{Corollary}
\newtheorem{pro}[thm]{Proposition}
\theoremstyle{definition}
\newtheorem{defi}[thm]{Definition}
\newtheorem{ex}[thm]{Example}
\newtheorem{rmk}[thm]{Remark}
\newtheorem{pdef}[thm]{Proposition-Definition}
\newtheorem{condition}[thm]{Condition}
\newtheorem{question}[thm]{Question}
\renewcommand{\labelenumi}{{\rm(\alph{enumi})}}
\renewcommand{\theenumi}{\alph{enumi}}

\nc{\ts}[1]{\textcolor{purple}{MTS:#1}}
\nc{\li}[1]{\textcolor{blue}{li:#1}}
\font\cyr=wncyr10

%%%%%%%%%%%%%%%%%%%%%%%%%%%%%%%%%%%%%%%%%%%%%%%%%%%%%%%%%%%%%%%%%%%%%%%%%%%%%%%%%%%%%%%%%%%%%%%%%%%%%%%%%%%%%%%%%%%%%%%%%%%%%%%%%%%%
 \title[]{\bf Transposed BiHom-Poisson algebras}

 \author[Ma]{Tianshui Ma\textsuperscript{*}}
 \address{School of Mathematics and Information Science, Henan Normal University, Xinxiang 453007, China}
         \email{matianshui@htu.edu.cn;matianshui@yahoo.com}

 \author[Li]{Bei Li}
 \address{School of Mathematics and Information Science, Henan Normal University, Xinxiang 453007, China}
         \email{libei9717@163.com}

  \thanks{\textsuperscript{*}Corresponding author}

\date{\today}

\begin{abstract} In this paper, we introduce the concept of transposed BiHom-Poisson (abbr. TBP) algebras which can be constructed by the BiHom-Novikov-Poisson algebras. Several useful identities for TBP algebras are provided. We also prove that the tensor product of two (T)BP algebras are closed. The notions of BP 3-Lie algebras and TBP 3-Lie algebras are presented and TBP algebras can induce TBP 3-Lie algebras by two approaches.  Finally, we give some examples for the TBP algebras of dimension 2.

\end{abstract}

 \keywords{transposed BiHom-Poisson algebras;BiHom-Novikov-Poisson algebras;transposed BiHom-Poisson 3-Lie algebras}

\subjclass[2020]{
17B61,%17B61 Hom-Lie and related algebras
17D30,%17D30 (non-Lie) Hom algebras and topics
%17B38,%17B38 Yang-Baxter equations and Rota-Baxter operators
17A30,%17A30 Nonassociative algebras satisfying other identities
%16T05,%16T05 Hopf algebras and their applications [See also 16S40, 57T05]
%16T10%16T10 Bialgebras
%16T25,%16T25 Yang-Baxter equations
%17B62,%17B62 Lie bialgebras; Lie coalgebras
17B63} %17B63 Poisson algebras

 \maketitle

\tableofcontents

\numberwithin{equation}{section}
\allowdisplaybreaks

%\newtheorem{definition}{Definition}[section]
%\newtheorem{theorem}{Theorem}[section]
%\newtheorem{lemma}{Lemma}[section]
%\newtheorem{remark}{Remark}[section]
%\newtheorem{proposition}{Proposition}[section]
%\newtheorem{corollary}{Corollary}[section]

%@=h长度i 设置元素间距
%@R=h长度i 设置行间距
%@C=h长度i 设置列间距
%@M=h长度i 设置元素的默认边距
%@W=h长度i 设置元素的默认宽度
%@H=h长度i 设置元素的默认高度
%@L=h长度i 设置标签的边距
%@C=3em@R=8ex@M=1em@W=0ex@H=1em@L=1ex

\section{Introduction and Preliminaries}
 Poisson algebras appear naturally in Hamiltonian mechanics, and play a central role in the study of Poisson geometry \cite{L}, noncommutative algebra \cite{V} and deformation quantization \cite{H}. Transposed Poisson algebras were presented in \cite{BBGW} by exchanging the operations $\cdot$ and $[,]$ in the compatible condition of Poisson algebras, and at the same time a factor 2 appears on the left-hand side. It was also studied in \cite{FKL,LS,YH} recently. A 3-Lie algebra is a vector space $A$ endowed with a ternary skew-symmetric operation satisfying the ternary Jacobi identity \cite{F,BCLM,BW}. The transposed Poisson algebra with a derivation can produce 3-Lie algebras \cite{BBGW}.

 Roughly speaking, a BiHom-associative algebra (or Lie algebra) is an algebra (or Lie algebra) such that the associativity (or Jacobi condition) is twisted by two (commuting) endomorphisms, for details see \cite{GMMP}, which can be seen as an extension of Hom-type algebra \cite{HLS} arising in quasi-deformations of Lie algebras of vector fields. Now there are so many research related to BiHom-type algebras, see refs\cite{BMN,GK,GZW,KMS,LS,LC,LCC,LMMP6,LMMP3,LMMP2,MLi,MLiY,MM,MY,MYZZ,TC}. In \cite{LMMP2}, the authors introduced the notion of BiHom-Poisson algebras and gave a necessary and sufficient condition under which BiHom-Novikov-Poisson algebras (which are twisted generalizations of Novikov-Poisson algebras \cite{X} and Hom-Novikov-Poisson algebras \cite{Yau11}) give rise to BiHom-Poisson algebras. Following \cite{KMS}, one can get 3-BiHom-Lie algebras which can be constructed from 3-Bihom-Lie algebras and 3-totally Bihom-associative algebras \cite{LC}. In \cite{BMN}, the authors studied the relationships between the BiHom-Lie superalgebras and its induced 3-BiHom-Lie superalgebras. In \cite{LS}, the authors introduced the notion of transposed Hom-Poisson algebra and studied the bimodule and matched pair of transposed Hom-Poisson algebras.

 It is very natural and interesting to consider the BiHom-deformation of transposed Poisson algebras and also study its properties, such as, the relations with BiHom-Novikov-Poisson algebras, BP algebras, 3-BiHom-Lie algebras and TBP 3-Lie algebra, the tensor product of two TBP algebras, etc. In this paper, we will investigate the questions above and beyond.

 The layout of the paper is as follows. In Section \ref{se:tbp}, we introduce the notion of TBP algebras (Definition \ref{de:5.1}) which extends transposed Poisson algebra defined in \cite{BBGW}. We note here when $\a=\b$, it is still different from the transposed Hom-Poisson algebra introduced in \cite{LS}. A TBP algebra can be constructed by a commutative associative algebra with a derivation (Proposition \ref{pro:5.3}). Several useful and interesting identities (Theorem \ref{thm:5.4}, Lemma \ref{lem:5.4-1}) are provided which will be used in Definition \ref{de:7.3}, Theorems \ref{thm:5.6}, \ref{thm:7.8}, \ref{thm:7.12}, \ref{thm:7.7} and Propositions \ref{pro:5.7}, \ref{pro:7.14}. We prove that the tensor product of two regular (T)BP algebras is closed (Lemma \ref{lem:5.5}, Theorem \ref{thm:5.6}) and a regular BiHom-Novikov-Poisson algebra can give rise to a TBP algebra (Theorem \ref{thm:6.2}). At the end of this section, we discuss the relations among BiHom-pre-Lie commutative algebras (Definition \ref{de:6.5}), differential BiHom-Novikov-Poisson algebras (Definition \ref{de:6.7}) and BiHom-pre-Lie Poisson algebras (Definition \ref{de:6.8}). In Section \ref{se:tbp3}, we focus on the research of 3-BiHom-Lie algebras and TBP 3-Lie algebras (Definition \ref{de:7.9}). Firstly we present two methods to construct 3-BiHom-Lie algebras from TBP algebras (Theorems \ref{thm:7.8}, \ref{thm:7.12}). We prove that a strong BP 3-Lie algebra can be obtained by a regular strong BP algebra with a derivation (Theorem \ref{thm:7.7}) and a 3-BiHom-Lie algebra can be obtained by a regular strong BP algebra (Proposition \ref{pro:7.14}). Secondly, a TBP 3-Lie algebra can be gotten by a strong BP algebra (Theorem \ref{thm:7.15}). In Section \ref{se:ex}, we give some examples of TBP algebras of dimension 2.

 Throughout this paper, $K$ will be a field of characteristic 0, and all vector spaces, tensor products, and homomorphisms are over $K$. And we denote by $\id_M$ the identity map from $M$ to $M$. First we recall some useful definitions which will be used later.

 \begin{defi}\label{de:4.1}\cite{GMMP,LMMP3} A {\bf BiHom-associative algebra} is a 4-tuple $(L,\cdot,\a,\b)$, where $L$ is a linear space, $\a,\b: L\lr L$ and $\cdot:L\o L\lr L$ are linear maps such that
 \begin{eqnarray}
 &\a \circ \b=\b \circ \a,\ \a(x\cdot y)=\a(x)\cdot \a(y),\ \b(x\cdot y)=\b(x)\cdot \b(y)&\label{eq:1.2}\\
 &\a(x)\cdot(y\cdot z)=(x\cdot y)\cdot \b(z),&\label{eq:4.1}
 \end{eqnarray}
 for all $x,\ y,\ z\in L$. The maps $\a$ and $\b$ (in this order) are the {\bf structure maps of $L$} and Eq.(\ref{eq:4.1}) is the {\bf BiHom-associativity condition}.

 A {\bf morphism $f:(A,\cdot_{A},\alpha_{A},\beta_{A})\longrightarrow (B,\cdot_{B},\alpha_{B},\beta_{B})$ of BiHom-associative algebras} is a linear map $f:A\longrightarrow B$ such that $\alpha_{B}\circ f=f\circ\alpha_{A}$, $\beta_{B}\circ f=f\circ\beta_{A}$ and $f\circ\cdot_{A}=\cdot_{B}\circ(f\otimes f)$.

 A BiHom-associative algebra $(L,\cdot,\a,\b)$ is {\bf BiHom-commutative} if
 \begin{eqnarray}
 &\b(x)\cdot\a(y)=\b(y)\cdot\a(x),\ \forall x,y\in L.&\label{eq:4.2}
 \end{eqnarray}
 \end{defi}

 \begin{rmk}\label{rmk:de:4.1}(1)\cite[Claim 3.7]{GMMP} Let $(L, \mu)$ be an associative algebra, $\a, \b: L\lr L$ two linear maps satisfying Eq.(\mref{eq:1.2}). Then $(L, \mu\ci (\a\o \b), \a, \b)$ is a BiHom-associative algebra.

 (2) Furthermore, if $(L, \mu)$ is commutative, then $(L, \mu\ci (\a\o \b), \a, \b)$ is BiHom-commutative.
 \end{rmk}

 \begin{defi}\label{de:4.4}\cite{GMMP,LMMP6} A {\bf BiHom-Lie algebra} $(L,[\cdot,\cdot],\a,\b)$ is a 4-tuple in which $L$ is a linear space, $\a, \b: L\lr L$ are two commuting linear maps and $[\cdot,\cdot]: L\times L\lr L$ is a bilinear map such that $\a(x\cdot y)=\a(x)\cdot \a(y),\ \b(x\cdot y)=\b(x)\cdot \b(y)$ and
 \begin{eqnarray}
 &[\b(x),\a(y)]=-[\b(y),\a(x)],\ \ (\hbox{BiHom-skew-symmetry})&\label{eq:4.4}\\
 &[[\b(x),\a(y)],\a^{2}(z)]+[[\b(y),\a(z)], \a^{2}(x)]+[[\b(z),\a(x)], \a^{2}(y)]=0\nonumber\\
 &\hbox{i.e.},\circlearrowleft_{x,y,z \in L}[[\b(x),\a(y)],\a^{2}(z)]=0~(\hbox{BiHom-Jacobi condition}),&\label{eq:4.7}
 \end{eqnarray}
 for all $x, y, z\in L$. The maps $\a$ and $\b$ (in this order) are the structure maps of $L$.
 \end{defi}

 \begin{rmk}\label{rmk:de:4.4}\cite[Remark 4.10]{LMMP6} Under the assumption of Definition \ref{de:4.4}. Assume that $\a, \b$ are bijective, the BiHom-Jacobi condition is equivalent to:
 \begin{eqnarray}
 &\circlearrowleft_{x,y,z \in L}[\b^{2}(x),[\b(y),\a(z)]]=0.&\label{eq:4.5}
 \end{eqnarray}
 \end{rmk}

 \begin{rmk}\label{rmk:de:4.4-1}\cite[Proposition 3.16]{GMMP} Let $(L, [,])$ be a Lie algebra, $\a, \b: L\lr L$ be two commuting linear maps such that $\a[x, y]=[\a(x), \a(y)], \b[x, y]=[\b(x), \b(y)]$. Then $(L, [\cdot,\cdot]'=[\cdot,\cdot]\circ (\a\o\b), \a, \b)$ is a BiHom-Lie algebra.
 \end{rmk}

 \begin{defi}\label{de:4.5}\cite{LMMP2} A {\bf BiHom-Poisson algebra} is a 5-tuple $(L,\cdot,[,],\a,\b)$ such that $(L,\cdot,\a,\b)$ is a BiHom-commutative algebra, $(L,[,],\a,\b)$ is a BiHom-Lie algebra and the BiHom-Leibniz rule holds, i.e.,
 \begin{eqnarray}
 &[\a\b(x),y\cdot z]=[\b(x),y]\cdot \b(z)+\b(y)\cdot[\a(x),z],&\label{eq:4.6}
 \end{eqnarray}
 for all $x, y, z\in L.$ If the structure maps $\a, \b$ are invertible, we call $(L, \cdot, [,], \a, \b)$ {\bf regular}.
 \end{defi}

 Let $(L, \mu)$ be an associative algebra (write $\mu(x\o y)=x y$). Unless otherwise specified, that $D:L\lr L$ is a derivation means $D(xy)=D(x)y+xD(y)$, $\forall~x, y\in L$. Similar definition for Lie algebras.

\section{Transposed BiHom-Poisson algebras}\label{se:tbp} The aim of this section is to investigate the properties of TBP algebras.
\subsection{Definition}
 \begin{defi}\label{de:5.1} A {\bf transposed BiHom-Poisson algebra} is a 5-tuple $(L, \cdot, [,], \a, \b)$ such that $(L, \cdot, \a$, $\b)$ is a BiHom-commutative algebra, $(L, [,], \a, \b)$ is a BiHom-Lie algebra and the following identity holds
 \begin{eqnarray}
 &2\a\b(x)\cdot[y,z]=[\b(x)\cdot y,\b(z)]+[\b(y),\a(x)\cdot z],&\label{eq:5.1}
 \end{eqnarray}
 for all $x, y, z\in L.$ If the structure maps $\a, \b$ are invertible, we call $(L, \cdot, [,], \a, \b)$ {\bf regular}.
 \end{defi}

 \begin{rmk} (1) If $\a=\b=\id$, then the TBP algebra is exactly the transposed Poisson algebra introduced in \cite{BBGW}.

 (2) Eq.(\ref{eq:5.1}) can be obtained by exchanging the two operations $\cdot$ and $[, ]$ in  Eq.(\ref{eq:4.6}) and at the same time the factor 2 appears at the left hand.
 %\li{exchanging the two operations $\cdot$ and $[, ]$ in  Eq.(\ref{eq:4.6}), we obtain $\a\b(x)\cdot[y,z]=[\b(x)\cdot y,\b(z)]+[\b(y),\a(x)\cdot z]$}
 \end{rmk}

 \begin{pro}\label{pro:5.2} Let $(L, \cdot, [,])$ be a transposed Poisson algebra,  $\a, \b: L\lr L$ be two commuting linear maps such that $\a(x\cdot y)=\a(x)\cdot \a(y), \b(x\cdot y)=\b(x)\cdot \b(y), \a[x, y]=[\a(x), \a(y)], \b[x, y]=[\b(x), \b(y)]$. Then $(L, \cdot'=\cdot\circ (\a\o\b), [\cdot,\cdot]'=[\cdot,\cdot]\circ (\a\o\b), \a, \b)$ is a TBP algebra, called the ``Yau twist" of $(L, \cdot, [,]).$
 \end{pro}

 \begin{proof} 
 We only check the compatible condition as follows. And others are obvious by Remarks \ref{rmk:de:4.1} and \ref{rmk:de:4.4-1}. For all $x, y, z\in L$, we have
 \begin{eqnarray*}
 [\b(x)\cdot' y, \b(z)]'+[\b(y), \a(x)\cdot' z]'
 &=&[\a^{2}\b(x)\cdot\a\b(y), \b^{2}(z)]+[\a\b(y), \a^{2}\b(x)\cdot\b^{2}(z)]\\
 &=&2\a^{2}\b(x)\cdot[\a\b(y), \b^2(z)]=2\a\b(x)\cdot' [y, z]',
 \end{eqnarray*}
 as we needed.
 \end{proof}

 \begin{pro}\label{pro:5.3} Let $(L, \mu)$ be a commutative associative algebra and $D:L\lr L$ a derivation. Assume that $\a, \b: L\lr L$ are two commuting algebra maps and any two of the maps $\a, \b, D$ commute. Define a Lie bracket on $L$ by
 \begin{eqnarray}
 &[x,y]=\a(x)D\b(y)-\b(y)D\a(x), \ \forall x, y\in L.&\label{eq:5.2}
 \end{eqnarray}
 Then $(L, \cdot, [,], \a, \b)$ is a TBP algebra, where $x\cdot y=\a(x)\b(y)$.
 \end{pro}

 \begin{proof} We only check the BiHom-Jacobi condition Eq.(\ref{eq:4.5}) and the compatible condition Eq.(\ref{eq:5.1}) as follows. For all $x, y, z\in L$, we have{\small
 \begin{eqnarray*}
 \circlearrowleft_{x,y,z \in L}[[\b(x),\a(y)],\a^{2}(z)]\hspace{-2mm}
 &=&\hspace{-2mm}\circlearrowleft_{x,y,z \in L}\Big(\a^{2}\b(x)D\a^{2}\b(y)D\a^{2}\b(z)-\a^{2}\b(z)\a^{2}\b(x)D^{2}\a^{2}\b(y)\\
 &&-\a^{2}\b(z)D\a^{2}\b(x)D\a^{2}\b(y)-\a^{2}\b(y)D\a^{2}\b(x)D\a^{2}\b(z)\\
 &&+\a^{2}\b(z)\a^{2}\b(y)D^{2}\a^{2}\b(x)+\a^{2}\b(z)D\a^{2}\b(y)D\a^{2}\b(x)\Big)=0
 \end{eqnarray*}}
 and{\small
 \begin{eqnarray*}
 [\b(x)\cdot y,\b(z)]+[\b(y),\a(x)\cdot z]\hspace{-2mm}
 &=&\hspace{-2mm}\a^{2}\b(x)\a\b(y)D\b^{2}(z)-\b^{2}(z)D\a^{2}\b(x)\a\b(y)-\b^{2}(z)\a^{2}\b(x)D\a\b(y)\\
 \hspace{-2mm}&&\hspace{-4mm}+\a\b(y)D\a^{2}\b(x)\b^{2}(z)+\a\b(y)\a^{2}\b(x)D\b^{2}(z)-\a^{2}\b(x)\b^{2}(z)D\a\b(y)\\
 \hspace{-2mm}&=&\hspace{-2mm}2\a^{2}\b(x)\a\b(y)D\b^{2}(z)-2\a^{2}\b(x)\b^{2}(z)D\a\b(y)\\
 \hspace{-2mm}&=&\hspace{-2mm}2\a\b(x)\cdot[y,z],
 \end{eqnarray*}}
 as desired.
 \end{proof}

\subsection{Several useful identities}
 The following theorem provides several useful and interesting identities.
 \begin{thm}\label{thm:5.4} Let $(L, \cdot, [,], \a, \b)$ be a TBP algebra. Then the following identities hold:
 \begin{eqnarray}
 &\circlearrowleft_{x,y,z \in L}\a\b^{2}(x)\cdot[\a\b(y),\a^{2}(z)]=0;& \label{eq:5.3}\\  
 &\circlearrowleft_{x,y,z \in L}[\a\b^{2}(h)\cdot[\a\b(x),\a^{2}(y)],\a^{3}\b(z)]=0;& \label{eq:5.4}\\
 &\circlearrowleft_{x,y,z \in L}[[\a\b^{2}(x),\a^{2}\b(y)],\a^{2}\b(z)\cdot\a^{3}(h)]=0;& \label{eq:5.5}\\
 &\circlearrowleft_{x,y,z \in L}[\a\b^{3}(x),\a^{2}\b^{2}(y)]\cdot[\a^{2}\b(h),\a^{3}\b(z)]=0.& \label{eq:5.6}
 \end{eqnarray}
 where $x, y, z, u, v\in L$.
 \end{thm}

 \begin{proof} We check the equalities as follows.

 {\bf Proof of Eq.(\ref{eq:5.3}):} We compute
 \begin{eqnarray*}
 2\a\b^{2}(x)\cdot[\a\b(y),\a^{2}(z)]&\stackrel{(\ref{eq:5.1})}=&[\b^{2}(x)\cdot \a\b(y),\a^{2}\b(z)]+[\a\b^{2}(y),\a\b(x)\cdot \a^{2}(z)]\\
 &\stackrel{(\ref{eq:4.2})(\ref{eq:4.4})}=&[\b^{2}(x)\cdot \a\b(y),\a^{2}\b(z)]-[\b^{2}(z)\cdot\a\b(x),\a^{2}\b(y)].
 \end{eqnarray*}
 Similarly,
 \begin{eqnarray*}
 &&2\a\b^{2}(y)\cdot[\a\b(z),\a^{2}(x)]=[\b^{2}(y)\cdot\a\b(z),\a^{2}\b(x)]
 -[\b^{2}(x)\cdot\a\b(y),\a^{2}\b(z)],\\
 &&2\a\b^{2}(z)\cdot[\a\b(x),\a^{2}(y)]=[\b^{2}(z)\cdot\a\b(x),\a^{2}\b(y)]
 -[\b^{2}(y)\cdot\a\b(z),\a^{2}\b(x)].
 \end{eqnarray*}
 By summing the above three formulas, we have Eq.(\ref{eq:5.3}).

 {\bf Proof of Eq.(\ref{eq:5.4}):} By Eq.(\ref{eq:5.1}), we can get
 \begin{eqnarray}
 &[\a\b^{2}(y),\a\b(h)\cdot\a^{2}(z)]=2\a\b^{2}(h)\cdot[\a\b(y),\a^{2}(z)]-[\b^{2}(h)\cdot \a\b(y),\a^{2}\b(z)].&\label{eq:5.9}
 \end{eqnarray}
 Applying BiHom-Jacobi condition Eq.(\ref{eq:4.7}), then
 \begin{eqnarray}
 &&[[\a\b^{2}(x),\a^{2}\b(y)],\a^{2}\b(z)\cdot\a^{3}(h)]+[[\a\b^{2}(y),\a\b(z)\cdot\a^{2}(h)],\a^{3}\b(x)]\nonumber\\
 &&\hspace{60mm}+[[\b^{2}(z)\cdot\a\b(h),\a^{2}\b(x)],\a^{3}\b(y)]=0.\label{eq:5.10}
 \end{eqnarray}
 By Eqs.(\ref{eq:4.2}) and (\ref{eq:5.9}),
 \begin{eqnarray}
 &&[[\a\b^{2}(x),\a^{2}\b(y)],\a^{2}\b(z)\cdot\a^{3}(h)]
 +2[\a\b^{2}(h)\cdot[\a\b(y),\a^{2}(z)],\a^{3}\b(x)]\nonumber\\
 &&-[[\b^{2}(h)\cdot \a\b(y),\a^{2}\b(z)],\a^{3}\b(x)]
 +[[\b^{2}(h)\cdot\a\b(z),\a^{2}\b(x)],\a^{3}\b(y)]=0.\label{eq:a-1}
 \end{eqnarray}
 Likewise,
 \begin{eqnarray}
 &&[[\a\b^{2}(y),\a^{2}\b(z)],\a^{2}\b(x)\cdot\a^{3}(h)]
 +2[\a\b^{2}(h)\cdot[\a\b(z),\a^{2}(x)],\a^{3}\b(y)]\nonumber\\
 &&-[[\b^{2}(h)\cdot\a\b(z),\a^{2}\b(x)],\a^{3}\b(y)]
 +[[\b^{2}(h)\cdot\a\b(x),\a^{2}\b(y)],\a^{3}\b(z)]=0,\label{eq:a-2}\\
 &&[[\a\b^{2}(z),\a^{2}\b(x)],\a^{2}\b(y)\cdot\a^{3}(h)]
 +2[\a\b^{2}(h)\cdot[\a\b(x),\a^{2}(y)],\a^{3}\b(z)]\nonumber\\
 &&-[[\b^{2}(h)\cdot\a\b(x),\a^{2}\b(y)],\a^{3}\b(z)]
 +[[\b^{2}(h)\cdot\a\b(y),\a^{2}\b(z)],\a^{3}\b(x)]=0.\label{eq:a-3}
 \end{eqnarray}
 By Eq.(\ref{eq:a-1}) plus Eq.(\ref{eq:a-2}) plus Eq.(\ref{eq:a-3}), one gets
 \begin{eqnarray}
 &&\hspace{-15mm}\circlearrowleft_{x,y,z \in L}[[\a\b^{2}(x),\a^{2}\b(y)],\a^{2}\b(z)\cdot\a^{3}(h)]
 +\circlearrowleft_{x,y,z \in L}2[\a\b^{2}(h)\cdot[\a\b(x),\a^{2}(y)],\a^{3}\b(z)]=0.\label{eq:5.10}
 \end{eqnarray}
 On the other hand, by Eqs.(\ref{eq:5.1}) and (\ref{eq:4.2}), we obtain
 \begin{eqnarray}
 &&2\a^{2}\b^{2}(h)\cdot[[\a\b(x),\a^{2}(y)],\a^{3}(z)]\nonumber\\
 &&\hspace{20mm}=[[\a\b^{2}(x),\a^{2}\b(y)],\a^{2}\b(z)\cdot\a^{3}(h)]
 +[\a\b^{2}(h)\cdot[\a\b(x),\a^{2}(y)],\a^{3}\b(z)],\label{eq:a-4}\\
 &&2\a^{2}\b^{2}(h)\cdot[[\a\b(y),\a^{2}(z)],\a^{3}(x)]\nonumber\\
 &&\hspace{20mm}=[[\a\b^{2}(y),\a^{2}\b(z)],\a^{2}\b(x)\cdot\a^{3}(h)]
 +[\a\b^{2}(h)\cdot[\a\b(y),\a^{2}(z)],\a^{3}\b(x)],\label{eq:a-5}\\
 &&2\a^{2}\b^{2}(h)\cdot[[\a\b(z),\a^{2}(x)],\a^{3}(y)]\nonumber\\
 &&\hspace{20mm}=[[\a\b^{2}(z),\a^{2}\b(x)],\a^{2}\b(y)\cdot\a^{3}(h)]
 +[\a\b^{2}(h)\cdot[\a\b(z),\a^{2}(x)],\a^{3}\b(y)].\label{eq:a-6}
 \end{eqnarray}
 Then the equality below can be obtained by Eq.(\ref{eq:a-4}) plus Eq.(\ref{eq:a-5}) plus Eq.(\ref{eq:a-6}) and by Eq.(\ref{eq:4.7}).
 \begin{eqnarray}
 &&\circlearrowleft_{x,y,z \in L}\Big([[\a\b^{2}(x),\a^{2}\b(y)],\a^{2}\b(z)\cdot\a^{3}(h)]\label{eq:5.11}
 +[\a\b^{2}(h)\cdot[\a\b(x),\a^{2}(y)],\a^{3}\b(z)]\Big)=0.
 \end{eqnarray}
 By Eq.(\ref{eq:5.11}) minus Eq.(\ref{eq:5.10}), one has Eq.(\ref{eq:5.4}).

 {\bf Proof of Eq.(\ref{eq:5.5}):} By Eq.(\ref{eq:5.11}) minus Eq.(\ref{eq:5.4}), one gets Eq.(\ref{eq:5.5}).

 {\bf Proof of Eq.(\ref{eq:5.6}):}
 By Eqs.(\ref{eq:5.1}) and (\ref{eq:4.2}), we obtain
 \begin{eqnarray}
 &&2[\a\b^{3}(x),\a^{2}\b^{2}(y)]\cdot[\a^{2}\b(h),\a^{3}\b(z)]\nonumber\\
 &&\hspace{15mm}=[\a\b^{2}(h)\cdot[\a\b^{2}(x),\a^{2}\b(y)],\a^{3}\b^{2}(z)]
 +[\a^{2}\b^{2}(h),\a^{2}\b^{2}(z)\cdot[\a^{2}\b(x),\a^{3}(y)],\quad\label{eq:a-7}\\
 &&2[\a\b^{3}(y),\a^{2}\b^{2}(z)]\cdot[\a^{2}\b(h),\a^{3}\b(x)]\nonumber\\
 &&\hspace{15mm}=[\a\b^{2}(h)\cdot[\a\b^{2}(y),\a^{2}\b(z)],\a^{3}\b^{2}(x)]
 +[\a^{2}\b^{2}(h),\a^{2}\b^{2}(x)\cdot[\a^{2}\b(y),\a^{3}(z)],\quad\label{eq:a-8}\\
 &&2[\a\b^{3}(z),\a^{2}\b^{2}(x)]\cdot[\a^{2}\b(h),\a^{3}\b(y)]\nonumber\\
 &&\hspace{15mm}=[\a\b^{2}(h)\cdot[\a\b^{2}(z),\a^{2}\b(x)],\a^{3}\b^{2}(y)]
 +[\a^{2}\b^{2}(h),\a^{2}\b^{2}(y)\cdot[\a^{2}\b(z),\a^{3}(x)].\quad\label{eq:a-9}
 \end{eqnarray}
 By Eq.(\ref{eq:a-7}) plus Eq.(\ref{eq:a-8}) plus Eq.(\ref{eq:a-9}) and Eqs.(\ref{eq:5.3}) and (\ref{eq:5.4}), Eq.(\ref{eq:5.6}) holds. We finish the proof.
 \end{proof}

 \begin{defi}\label{de:7.3} We call a BP algebra $(L, \cdot, [,], \a, \b)$ {\bf strong} if Eq.(\ref{eq:5.6}) holds.
 \end{defi}

 \begin{rmk} Eq.(\ref{eq:5.6}) holds for any TBP algebras.
 \end{rmk}

 The following result shows that the relation between BP algebras and TBP algebras.
 \begin{pro}\label{pro:5.7} Let $(L, \cdot, \a, \b)$ be a BiHom-commutative algebra and $(L, [,], \a, \b)$ be a BiHom-Lie algebra. Assume that $(L, \cdot, [, ], \a, \b)$ is both a BP algebra and a TBP algebra. Then
 \begin{eqnarray}
 &\a\b^{2}(z)\cdot[\a\b(x),\a^{2}(y)]=[\a\b^{2}(x),\a\b(z)\cdot \a^{2}(y)]=0.&\label{eq:5.14}
 \end{eqnarray}
 \end{pro}

 \begin{proof} For all $x, y, z\in L$, we have
 \begin{eqnarray}
 [\b^{2}(z)\cdot\a\b(x),\a^{2}\b(y)]\hspace{-3mm}&\stackrel{(\ref{eq:4.4})(\ref{eq:4.6})}=&
 \hspace{-3mm}
 -[\b^{2}(y),\a\b(z)]\cdot \a^{2}\b(x)-\a\b^{2}(z)\cdot[\a\b(y),\a^{2}(x)]\nonumber\\
 \hspace{-4mm}&\stackrel{(\ref{eq:4.2})}=&\hspace{-6mm}-\a\b^{2}(x)\cdot[\a\b(y),\a^{2}(z)]
 -\a\b^{2}(z)\cdot[\a\b(y),\a^{2}(x)],\label{eq:5.17}
 \end{eqnarray}
 and
 \begin{eqnarray}
 [\a\b^{2}(x),\a\b(z)\cdot \a^{2}(y)]\hspace{-4mm}&\stackrel{(\ref{eq:4.6})}=&\hspace{-4mm}\a\b^{2}(z)\cdot[\a\b(x),\a^{2}(y)]
 +[\b^{2}(x),\a\b(z)]\cdot \a^{2}\b(y)\nonumber\\
 &\stackrel{(\ref{eq:4.2})(\ref{eq:4.4})}=&\hspace{-4mm}\a\b^{2}(y)\cdot[\a\b(x),\a^{2}(z)]
 -\a\b^{2}(z)\cdot[\a\b(y),\a^{2}(x)].\label{eq:5.18}
 \end{eqnarray}
 Then
 \begin{eqnarray*}
 0&\stackrel{(\ref{eq:5.1})}=&[\b^{2}(z)\cdot \a\b(x),\a^{2}\b(y)]+[\a\b^{2}(x),\a\b(z)\cdot \a^{2}(y)]-2\a\b^{2}(z)\cdot[\a\b(x),\a^{2}(y)]\\
 &\stackrel{(\ref{eq:5.17})(\ref{eq:5.18})}=&-\a\b^{2}(x)\cdot[\a\b(y),\a^{2}(z)]\cdot\a^{2}\b(x)-\a\b^{2}(z)\cdot[\a\b(y),\a^{2}(x)]\\
 &&+\a\b^{2}(y)\cdot[\a\b(x),\a^{2}(z)]-\a\b^{2}(z)\cdot[\a\b(y),\a^{2}(x)]-2\a\b^{2}(z)\cdot[\a\b(x),\a^{2}(y)]\\
 &\stackrel{(\ref{eq:4.4})}=&-\a\b^{2}(x)\cdot [\a\b(y),\a^{2}(z)]-\a\b^{2}(y)\cdot[\a\b(z),\a^{2}(x)]\\
 &\stackrel{(\ref{eq:5.3})}=&\a\b^{2}(z)\cdot[\a\b(x),\a^{2}(y)].
 \end{eqnarray*}
 By Eq.(\ref{eq:5.18}) and the above identity, one gets $[\a\b^{2}(x),\a\b(z)\cdot \a^{2}(y)]=0$.
 \end{proof}

 \begin{rmk}\label{rmk:5.7a} (1) Assume that the structure maps $\a, \b$ are bijective. Then Eq.(\ref{eq:5.14}) is equivalent to
 \begin{eqnarray}
 &z\cdot[x,y]=[x,z\cdot y]=0.&\label{eq:5.14a}
 \end{eqnarray}
 In this case for Proposition \ref{pro:5.7}, the condition Eq.(\ref{eq:5.14a}) is also sufficient.

 (2) The intersection of TBP algebras and BP algebras is almost empty.
 \end{rmk}

\subsection{Tensor products of TBP algebras} In this subsection, we will prove that the tensor product of two (T)BP algebras is closed.
 \begin{lem}\label{lem:5.5} Let $(A,\cdot_{A},[,]_{A},\a_{A},\b_{A})$ and $(B,\cdot_{B},[,]_{B},\a_{B},\b_{B})$ be two regular BP algebras. Define two operations $\cdot$ and $[,]$ on $A\o B$ by
 \begin{eqnarray}
 &(a\o b)\cdot (a'\o b')=a\cdot_{A} a'\o b\cdot_{B}b',&\label{eq:5.12}\\
 &[a\o b,a'\o b']=[a,a']_{A}\o b\cdot_{B}b'+a\cdot_{A}a'\o [b,b']_{B},&\label{eq:5.13}
 \end{eqnarray}
 for all $a,a'\in A,\ b,b'\in B$. Then $(A\o B,\cdot,[,],\a_{A}\o\a_{B},\b_{A}\o\b_{B})$ is a BP algebra.
 \end{lem}

 \begin{proof} {\bf Step 1.} By \cite[Theorem 2.1]{LMMP2} we get that $(A\o B, \cdot, \a_{A}\o\a_{B}, \b_{A}\o\b_{B})$ is a BiHom-commutative algebra.\\
 {\bf Step 2.} We prove that $(A\o B,\cdot,[,],\a_{A}\o\a_{B},\b_{A}\o\b_{B})$ is a BiHom-Lie algebra. For all $a, a', a''\in A,\ b, b', b''\in B$, it is easy to find that $[(\b_{A}\o\b_{B})(a\o b),(\a_{A}\o\a_{B})(a'\o b')]=-[(\b_{A}\o\b_{B})(a'\o b'),(\a_{A}\o\a_{B})(a\o b)]$. The BiHom-Jacobi condition can be checked as follows.{\small
 \begin{eqnarray*}
 &&[[(\b_{A}\o\b_{B})(a\o b),(\a_{A}\o\a_{B})(a'\o b')],(\a{^{2}}_{A}\o\a{^{2}}_{B})(a''\o b'')]+[[(\b_{A}\o\b_{B})(a'\o b'),\\
 &&(\a_{A}\o\a_{B})(a''\o b'')],(\a{^{2}}_{A}\o\a{^{2}}_{B})(a\o b)]+[[(\b_{A}\o\b_{B})(a''\o b''),(\a_{A}\o\a_{B})(a\o b)],\\
 &&(\a{^{2}}_{A}\o\a{^{2}}_{B})(a'\o b')]=\Big([[\b_{A}(a),\a_{A}(a')],\a{^{2}}_{A}(a'')]\o(\b_{B}(b)\cdot \a_{B}(b'))\cdot \a{^{2}}_{B}(b'')\\
 &&+[[\b_{A}(a'),\a_{A}(a'')],\a{^{2}}_{A}(a)]\o(\b_{B}(b')\cdot \a_{B}(b''))\cdot \a{^{2}}_{B}(b)+[[\b_{A}(a''),\a_{A}(a)],\a{^{2}}_{A}(a')]\\
 &&\o(\b_{B}(b'')\cdot \a_{B}(b))\cdot \a{^{2}}_{B}(b')\Big)+\Big((\b_{A}(a)\cdot \a_{A}(a'))\cdot\a{^{2}}_{A}(a'')\o [[\b_{B}(b),\a_{B}(b')],\a{^{2}}_{B}(b'')]\\
 &&+(\b_{A}(a')\cdot \a_{A}(a''))\cdot\a{^{2}}_{A}(a)\o [[\b_{B}(b'),\a_{B}(b'')],\a{^{2}}_{B}(b)]+(\b_{A}(a'')\cdot \a_{A}(a))\cdot\a{^{2}}_{A}(a')\\
 &&\o [[\b_{B}(b''),\a_{B}(b)],\a{^{2}}_{B}(b')]\Big)+\Big([\b_{A}(a),\a_{A}(a')]\cdot \a{^{2}}_{A}(a'')\o[\b_{B}(b)\cdot \a_{B}(b'),\a{^{2}}_{B}(b'')]\\
 &&+[\b_{A}(a'),\a_{A}(a'')]\cdot \a{^{2}}_{A}(a)\o[\b_{B}(b')\cdot \a_{B}(b''),\a{^{2}}_{B}(b)]+[\b_{A}(a''),\a_{A}(a)]\cdot \a{^{2}}_{A}(a')\\
 &&\o[\b_{B}(b'')\cdot \a_{B}(b),\a{^{2}}_{B}(b')]+[\b_{A}(a)\cdot \a_{A}(a'),\a{^{2}}_{A}(a'')]\o [\b_{B}(b),\a_{B}(b')]\cdot \a{^{2}}_{B}(b'')\\
 &&+[\b_{A}(a')\cdot \a_{A}(a''),\a{^{2}}_{A}(a)]\o [\b_{B}(b'),\a_{B}(b'')]\cdot \a{^{2}}_{B}(b)+[\b_{A}(2a'')\cdot \a_{A}(a),\a{^{2}}_{A}(a')]\\
 &&\o [\b_{B}(b''),\a_{B}(b)]\cdot \a{^{2}}_{B}(b')\Big)\stackrel{\bigtriangleup}{=}\hbox{(I)+(II)+(III)}.
 \end{eqnarray*}}
 By Eqs.(\ref{eq:4.1}), (\ref{eq:4.2}) and (\ref{eq:4.7}), (I)=(II)=0.

 By Eqs.(\ref{eq:4.6}) and (\ref{eq:4.2}), we have
 \begin{eqnarray}
 &[\b(a')\cdot\a(a''),\a^{2}(a)]=-[\b(a),\a(a')]\cdot\a^{2}(a'')
 -[\b(a),\a(a'')]\cdot\a^{2}(a').&\label{eq:5.19}
 \end{eqnarray}
 Then {\small
 \begin{eqnarray*}
 \hbox{(III)}&\hspace{-2mm}\stackrel{(\ref{eq:5.19})}=&\hspace{-2mm}([\b_{A}(a),\a_{A}(a')]\cdot \a{^{2}}_{A}(a'')\o-([\b_{B}(b''),\a_{B}(b)]\cdot\a^{2}_{B}(b')+[\b_{B}(b''),\a_{B}(b')]\cdot\a^{2}_{B}(b))\\
 &&\hspace{-2mm}+[\b_{A}(a'),\a_{A}(a'')]\cdot \a{^{2}}_{A}(a)\o-([\b(_{B}b),\a_{B}(b')]\cdot\a^{2}_{B}(b'')+[\b_{B}(b),\a_{B}(b'')]\cdot\a^{2}_{B}(b'))\\
 &&\hspace{-2mm}+[\b_{A}(a''),\a_{A}(a)]\cdot \a{^{2}}_{A}(a')\o-([\b_{B}(b'),\a_{B}(b'')]\cdot\a^{2}_{B}(b)+[\b_{B}(b'),\a_{B}(b)]\cdot\a^{2}_{B}(b''))\\
 &&\hspace{-2mm}-([\b(a''),\a(a)]\cdot\a^{2}_{A}(a')+[\b_{A}(a''),\a_{A}(a')]\cdot\a^{2}_{A}(a))\o [\b_{B}(b),\a_{B}(b')])\cdot \a{^{2}}_{B}(b'')\\
 &&\hspace{-2mm}-([\b_{A}(a),\a_{A}(a')]\cdot\a^{2}_{A}(a'')+[\b_{A}(a),\a_{A}(a'')]\cdot\a^{2}_{A}(a'))\o [\b_{B}(b'),\a_{B}(b'')]\cdot \a{^{2}}_{B}(b)\\
 &&\hspace{-2mm}-([\b_{A}(a'),\a_{A}(a'')]\cdot\a^{2}_{A}(a)+[\b_{A}(a'),\a_{A}(a)]\cdot\a^{2}_{A}(a''))\o [\b_{B}(b''),\a_{B}(b)]\cdot \a{^{2}}_{B}(b'))\\
 &\stackrel{(\ref{eq:4.4})}=&\hspace{-2mm}0.
 \end{eqnarray*}}
 {\bf Step 3.} We prove the compatible condition Eq.(\ref{eq:4.6}). {\small
 \begin{eqnarray*}
 &&\hspace{-5mm}[(\b_{A}\o\b_{B})(a\o b),a'\o b']\cdot(\b_{A}\o\b_{B})(a''\o b'')+(\b_{A}\o\b_{B})(a'\o b')\cdot[(\a_{A}\o\a_{B})(a\o b),a''\o b'']\\
 &&\stackrel{(\ref{eq:4.1})(\ref{eq:4.2})}=[[\b_{A}(a),a']\cdot \b_{A}(a'')\o \a_{B}\b_{B}(b)\cdot (b'\cdot b'')+\a_{A}\b_{A}(a)\cdot (a'\cdot a'')\o[\b_{B}(b), b']\cdot\b_{B}(b'')\\
 &&\hspace{10mm}+\b_{A}(a')\cdot[\a_{A}(a),a'']\o\a_{B}\b_{B}(b)\cdot (b'\cdot b'')+\a_{A}\b_{A}(a)\cdot (a'\cdot a'')\o\b_{B}(b')\cdot[\a_{B}(b), b'']\\
 &&\hspace{3mm}\stackrel{(\ref{eq:4.6})}=[\a_{A}\b_{A}(a),a'\cdot a'']\o\a_{B}\b_{B}(b)\cdot (b'\cdot b'')+\a_{A}\b_{A}(a)\cdot (a'\cdot a'')\o[\a_{B}\b_{B}(b),b'\cdot b'']\\
 &&\hspace{4mm}=[\a_{A}\b_{A}(a)\o\a_{B}\b_{B}(b),a'\cdot a''\o b'\cdot b''],
 \end{eqnarray*}}
 finishing the proof. Therefore, $(A\o B, \cdot, [,], \a_{A}\o\a_{B}, \b_{A}\o\b_{B})$ is a BP algebra.
 \end{proof}

 \begin{thm}\label{thm:5.6} Let $(A,\cdot_{A},[,]_{A},\a_{A},\b_{A})$ and $(B,\cdot_{B},[,]_{B},\a_{B},\b_{B})$ be two regular TBP algebras. Define two operations $\cdot$ and $[, ]$ on $A\o B$ by Eqs.(\ref{eq:5.12}) and (\ref{eq:5.13}) respectively. Then $(A\o B, \cdot, [,], \a_{A}\o\a_{B}, \b_{A}\o\b_{B})$ is a TBP algebra.
 \end{thm}

 \begin{proof} Based on the proof of Lemma \ref{lem:5.5}, we only need to verify the following two assertions: (i) The expression (III) in Lemma \ref{lem:5.5} equals 0; (ii) Eq.(\ref{eq:5.1}) holds for $(A\o B, \cdot, [,], \a_{A}\o\a_{B}, \b_{A}\o\b_{B})$.

 {\bf Proof of (i).} By Eq.(\ref{eq:5.3}), we have
 \begin{eqnarray}
 &[\b(a'),\a(a'')]\cdot\a^{2}(a)=-[\b(a''),\a(a)]\cdot\a^{2}(a')
 -[\b(a),\a(a')]\cdot\b^{2}(a'').&\label{eq:a-10}
 \end{eqnarray}
 Now we denote the first three items of (III) by (IV) and the last three items of (III) by (V). Then{\small
 \begin{eqnarray*}
 (\hbox{IV})\hspace{-3mm}&\stackrel{(\ref{eq:a-10})(\ref{eq:4.4})}=&\hspace{-3mm}
 [\b_{A}(a),\a_{A}(a')]\cdot \a{^{2}}_{A}(a'')\o[\b_{B}(b)\cdot \a_{B}(b')\cdot\a{^{2}}_{B}(b'')]+\Big([\b_{A}(a''),\a_{A}(a)]\cdot\a^{2}_{A}(a')\\
 &&\hspace{-6mm}+[\b_{A}(a),\a_{A}(a')]\cdot\a{^{2}}_{A}(a'')\Big)\o[\a\b_{B}(b), \a_{B}(b'),\a{^{2}}_{B}\b_{B}^{-1}(b)]+[\b_{A}(a''),\a_{A}(a)]\cdot \a{^{2}}_{A}(a')\\
 &&\hspace{-6mm}\o[\b_{B}(b'')\cdot \a_{B}(b),\a{^{2}}_{B}(b')]\\
 &\hspace{-3mm}\stackrel{(\ref{eq:5.1})}=&\hspace{-8mm}[\b_{A}(a),\a_{A}(a')]
 \cdot\a{^{2}}_{A}(a'')\o 2\a_{B}\b_{B}(b')\cdot[\a_{B}(b),\a{^{2}}_{B}\b_{B}^{-1}(b'')]
 +[\b_{A}(a''),\a_{A}(a)]\cdot \a{^{2}}_{A}(a')\\
 &&\hspace{-6mm}\o 2\a_{B}\b_{B}(b'')\cdot[\a_{B}(b),\a{^{2}}_{B}\b_{B}^{-1}(b')].
 \end{eqnarray*}}
 Similarly,
 \begin{eqnarray*}
 (\hbox{V})\hspace{-3mm}&\stackrel{}=&\hspace{-3mm}2\a_{A}\b_{A}(a')\cdot[\a_{A}(a),\a{^{2}}_{A}\b{^{-1}}_{A}(a'')]\o[\b_{B}(b),\a_{B}(b')]\cdot \a{^{2}}_{B}(b'')+2\a_{A}\b_{A}(a'')\cdot[\a_{A}(a),\\
 &&\hspace{-3mm}\a{^{2}}_{A}\b{^{-1}}_{A}(a')]\o[\b_{B}(b''),\a_{B}(b)]\cdot \a{^{2}}_{B}(b').
 \end{eqnarray*}
 So (IV)+(V)=0 by Eq.(\ref{eq:4.2}), i.e., (III)=0, therefore $(A\o B, [,], \a, \b)$ is a BiHom-Lie algebra.

 {\bf Proof of (ii).}{\small
 \begin{eqnarray*}
 &&\hspace{-15mm}[\b(a'\o b'),\a(a\o b)\cdot(a''\o b'')]+[\b(a\o b))\cdot(a'\o b'),\b(a''\o b'')]\\
 &=&\hspace{-5mm}[\b_{A}(a'),\a_{A}(a)\cdot a'']\o\b_{B}(b')\cdot (\a_{B}(b)\cdot b'')+\b_{A}(a')\cdot(\a_{A}(a)\cdot a'')\o[\b_{B}(b'),\a_{B}(b)\cdot b'']\\
 &&\hspace{-5mm}+[\b_{A}(a)\cdot a',\b_{A}(a'')]\o(\b_{B}(b)\cdot b')\cdot \b_{B}(b'')+(\b_{A}(a)\cdot a')\cdot\b_{A}(a'')\o[\b_{B}(b)\cdot b',b_{B}(b'')]\\
 &\stackrel{(\ref{eq:4.1})(\ref{eq:4.2})}=&\hspace{-3mm}([\b_{A}(a)\cdot a',\b_{A}(a'')]+[\b_{A}(a'),\a_{A}(a)\cdot a''])\o(\b_{B}(b)\cdot b')\cdot \b_{B}(b'')+(\b_{A}(a)\cdot a')\cdot\b_{A}(a'')\\
 &&\hspace{-3mm}\o([\b_{B}(b'),\a_{B}(b)\cdot b'']+[\b_{B}(b)\cdot b',b_{B}(b'')])\\
 &\stackrel{(\ref{eq:5.1})}=&\hspace{-3mm}2\a_{A}\b_{A}(a)\cdot[\b_{A}(a'),\b_{A}(a'')]\o(\b_{B}(b)\cdot b')\cdot \b_{B}(b'')+2(\b_{A}(a)\cdot \b_{A}(a'))\cdot\b_{A}(a'')\\
 &&\hspace{-4mm}\o\a_{B}\b_{B}(b)\cdot[b_{B}(b'),b_{B}(b'')]\\
 &\stackrel{(\ref{eq:4.1})}=&\hspace{-3mm}2\a_{A}\b_{A}(a)\cdot[\b_{A}(a'),\b_{A}(a'')]\o\a_{B}\b_{B}(b)\cdot (b'\cdot b'')+2\a_{A}\b_{A}(a)\cdot (\b_{A}(a')\cdot a'')\\
 &&\hspace{-4mm}\o\a_{B}\b_{B}(b)\cdot[b_{B}(b'),b_{B}(b'')]=2\a\b(a\o b)\cdot[a'\o b',a''\o b''],
 \end{eqnarray*}}
 as desired.
 \end{proof}

\subsection{Transposed BP algebras from BiHom-Novikov-Poisson algebras} Let us recall the definition of BiHom-Novikov-Poisson algebra.
 \begin{defi}\cite{LMMP3}\label{defi:6.1} A {\bf BiHom-Novikov-Poisson algebra} is a 5-tuple $(L, \cdot, \ast, \a, \b)$ such that

 (1) $(L, \cdot, \a, \b)$ is a BiHom-commutative algebra;

 (2) $(L, \ast, \a, \b)$ is a BiHom-Novikov algebra, i.e.,
 \begin{eqnarray}
 &(\b(x)\ast\a(y))\ast\b(z)-\a\b(x)\ast(\a(y)\ast z)=(\b(y)\ast\a(x))\ast\b(z)-\a\b(y)\ast(\a(x)\ast z);&\label{eq:6.1}\\
 &(x\ast\b(y))\ast\a\b(z)=(x\ast\b(z))\ast\a\b(y);&\label{eq:6.2}
 \end{eqnarray}

 (3) the following compatibility conditions hold:
 \begin{eqnarray}
 &(\b(x)\ast\a(y))\cdot\b(z)-\a\b(x)\ast(\a(y)\cdot z)=(\b(y)\ast\a(x))\cdot\b(z)-\a\b(y)\ast(\a(x)\cdot z);&\label{eq:6.3}\\
 &(x\cdot\b(y))\ast\a\b(z)=(x\ast\b(z))\cdot\a\b(y), x, y, z\in L.&\label{eq:6.4}
 \end{eqnarray}

 A BiHom-Novikov-Poisson algebra $(L, \cdot, \ast, \a, \b)$ is {\bf regular} if $\a, \b$ are invertible.
 \end{defi}

 \begin{rmk} The 4-tuple $(L, \ast, \a, \b)$ is called a {\bf BiHom-pre-Lie algebra} if Eq.(\ref{eq:6.1}) holds.
 \end{rmk}

 \begin{thm}\label{thm:6.2} Let $(L, \cdot, \ast, \a, \b)$ be a regular BiHom-Novikov-Poisson algebra. Define
 \begin{eqnarray}
 &[x,y]=x\ast y-\a^{-1}\b(y)\ast\a\b^{-1}(x),\ x,y\in L.&\label{eq:6.5}
 \end{eqnarray}
 Then $(L, \cdot, [,], \a, \b)$ is a TBP algebra.
 \end{thm}

 \begin{proof} By \cite[Proposition 3.4]{LMMP6}, $(L,[,],\a,\b)$ is a BiHom-Lie algebra. So we only need to prove Eq.(\ref{eq:5.1}). By \cite[Lemma 3.2]{LMMP3}, Eq.(\ref{eq:6.4}) is equivalent to
 \begin{eqnarray}
 &\a(x)\cdot(y\ast z)=(x\cdot y)\ast \b(z), x, y, z\in L. \label{eq:6.8}
 \end{eqnarray}
 Then for all $x, y, z\in L$, we have
 \begin{eqnarray*}
 &&\hspace{-15mm}[\b(x)\cdot y,\b(z)]+[\b(y),\a(x)\cdot z]-2\a\b(x)\cdot[y,z]\\
 &=&\hspace{-4mm}(\b(x)\cdot y)\ast\b(z)-\a^{-1}\b^{2}(z)\ast(\a(x)\cdot \a\b^{-1}(y))+\b(y)\ast(\a(x)\cdot z)-(\b(x)\cdot \a^{-1}\b(z))\ast\a(y)\\
 &&\hspace{-4mm}-2\a\b(x)\cdot(y\ast z-\a^{-1}\b(z)\ast\a\b^{-1}(y))\\
 &\stackrel{(\ref{eq:4.2})}=&\hspace{-4mm}((\b(x)\cdot y)\ast\b(z)-\a\b(x)\cdot(y\ast z)+\a\b(x)\cdot(\a^{-1}\b(z)\ast\a\b^{-1}(y))-(\b(x)\cdot \a^{-1}\b(z))\ast\a(y))\\
 &&\hspace{-4mm}-(\a^{-1}\b(y)\ast \a^{-1}\b(z))\cdot\a^{2}(x)-(\a^{-2}\b^{2}(z)\ast y)\cdot\a^{2}(x)-\b(y)\ast(\a^{-1}\b(z)\cdot \a^{2}\b^{-1}(x))\\
 &&\hspace{-4mm}+\a^{-1}\b^{2}(z)\ast(y\cdot \a^{2}\b^{-1}(x))\\
 &\stackrel{(\ref{eq:6.3})(\ref{eq:6.8})}=&\hspace{-3mm}0,
 \end{eqnarray*}
 finishing the proof.
 \end{proof}

\subsection{Three algebraic structures related to TBP algebras}
 \begin{defi}\label{de:6.5} A {\bf BiHom-pre-Lie commutative algebra} is a 5-tuple $(L, \cdot, \ast, \a, \b)$, where $(L, \cdot, \a, \b)$ is a BiHom-commutative algebra and $(L, \ast, \a, \b)$ is a BiHom-pre-Lie algebra satisfying
 \begin{eqnarray}
 &\a\b(x)\ast(\a(y)\cdot z)=(\b(x)\ast \a(y))\cdot \b(z)+\a\b(y)\cdot(\a(x)\ast z),\ x, y, z\in L.&\label{eq:6.6}
 \end{eqnarray}
 \end{defi}

 \begin{rmk}
 When $\a=\b=\id$ in Definition \ref{de:6.5}, we cover the pre-Lie commutative algebra in \cite{BBGW}.
 \end{rmk}

 \begin{pro}\label{pro:de:6.5} Let $(L, \cdot, \ast)$ be a pre-Lie commutative algebra, $\a, \b: L\lr L$ be two commuting linear maps such that $\a(x\cdot y)=\a(x)\cdot \a(y), \b(x\cdot y)=\b(x)\cdot \b(y), \a(x\ast y)=\a(x)\ast \a(y), \b(x\ast y)=\b(x)\ast \b(y)$. Then $(L, \cdot'=\cdot\circ (\a\o \b), \ast'=\ast \circ (\a\o \b),\a,\b)$ is a BiHom-pre-Lie commutative algebra, called the ``Yau twist" of $(L, \cdot, \ast)$.
 \end{pro}

 \begin{proof} By \cite[Proposition 3.2]{LMMP2}, $(L, \cdot'=\cdot\circ (\a\o \b), \ast'=\ast \circ (\a\o \b),\a,\b)$ is a BiHom-pre-Lie algebra. Furthermore, we have
 \begin{eqnarray*}
 (\b(x)\ast' \a(y))\cdot' \b(z)+\a\b(y)\cdot'(\a(x)\ast' z)\hspace{-2mm}
 &=&\hspace{-2mm}(\a^{2}\b(x)\ast \a^{2}\b(y))\cdot \b^{2}(z)+\a^{2}\b(y)\cdot(\a^{2}\b(x)\ast \b^{2}(z))\\
 &=&\hspace{-2mm}\a\b(x)\ast'(\a(y)\cdot' z).
 \end{eqnarray*}
 Therefore, $(L, \cdot'=\cdot\circ (\a\o \b), \ast'=\ast \circ (\a\o \b),\a,\b)$ is a BiHom-pre-Lie commutative algebra by Remark \ref{rmk:de:4.1}.
 \end{proof}

 \begin{pro}\label{pro:6.6} Let $(L, \mu)$ be a commutative associative algebra, $\a, \b : L\lr L$ be two algebra morphisms and $D: L\lr L $ be a derivation such that any two of the maps $\a, \b, D$ commute. Then $(L, \cdot, \ast, \a, \b)$ is a BiHom-pre-Lie commutative algebra, where $x\ast y=\a(x)D\b(y)$ and $x\cdot y=\a(x)\b(y)$, $\forall~~x,y\in L$.
 \end{pro}

 \begin{proof} By \cite[Corollary 2.10]{LMMP3}, $(L, \ast, \a, \b)$ is a BiHom-pre-Lie algebra. Further,
 \begin{eqnarray*}
 \a\b(x)\ast(\a(y)\cdot z)&=&\a^{2}\b(x)\a^{2}\b(y)D\b^{2}(z)+\a^{2}\b(x)D\a^{2}\b(y)\b^{2}(z)\\
 &=&(\b(x)\ast \a(y))\cdot \b(z)+\a\b(y)\cdot(\a(x)\ast z).
 \end{eqnarray*}
 Therefore, $(L, \cdot, \ast, \a, \b)$ is a BiHom-pre-Lie commutative algebra by Remark \ref{rmk:de:4.1}.
 \end{proof}

 \begin{defi}\label{de:6.7} A {\bf differential BiHom-Novikov-Poisson algebra} is a 5-tuple $(L, \cdot, \ast, \a, \b)$ such that (1) $(L, \cdot, \a, \b)$ is a BiHom-commutative algebra; (2) $(L, \ast, \a, \b)$ is a BiHom-Novikov algebra; (3) Eqs.(\ref{eq:6.4}) and (\ref{eq:6.6}) hold.
 \end{defi}

 \begin{rmk}\label{rmk:de:6.7} (1) When $\a=\b=\id$ in Definition \ref{de:6.7}, we call $(L, \cdot, \ast)$ a {\bf differential Novikov-Poisson algebra}.

 (2) The differential Novikov-Poisson algebra here is slightly different from the differential Novikov-Poisson algebra introduced in \cite{BCZ} and the former can induce the later. In fact, for all $x, y, z\in L$, we have
 \begin{eqnarray*}
 &&\hspace{-15mm}(\b(x)\ast\a(y))\cdot\b(z)-\a\b(x)\ast(\a(y)\cdot z)-(\b(y)\ast\a(x))\cdot\b(z)+\a\b(y)\ast(\a(x)\cdot z)\\
 &\stackrel{(\ref{eq:6.6})}=&(\b(x)\ast\a(y))\cdot\b(z)-(\b(x)\ast\a(y))\cdot\b(z)-\a\b(y)\cdot(\a(x)\ast z)\\
 &&-(\b(y)\ast\a(x))\cdot\b(z)+(\b(y)\ast\a(x))\cdot\b(z)+\a\b(x)\cdot(\a(y)\ast z)\\
 &\stackrel{(\ref{eq:6.8})}=&0.
 \end{eqnarray*}
 Then Eq.(\ref{eq:6.3}) holds.

 (3) By \cite[Proposition 3.5]{LMMP3} and Proposition \ref{pro:de:6.5}, the ``Yau twist" for differential Novikov-Poisson algebras holds.
 \end{rmk}

 \begin{defi}\label{de:6.8} A {\bf BiHom-pre-Lie Poisson algebra} is a 5-tuple $(L, \cdot, \ast, \a, \b)$, where $L$ is a vector space and $\cdot, \ast$ are two bilinear operations on $L$ such that (1) $(L, \cdot, \a, \b)$ is a BiHom-commutative algebra; (2) $(L, \ast, \a, \b)$ is a BiHom-pre-Lie algebra; (3) Eqs. (\ref{eq:6.3}) and (\ref{eq:6.4}) hold.
 
 \end{defi}

 \begin{rmk}(1) It is obvious that a BiHom-Novikov-Poisson algebra is a BiHom-pre-Lie Poisson algebra.

 (2) By \cite[Proposition 3.5]{LMMP3}, the ``Yau twist" for BiHom-pre-Lie Poisson algebra still holds.

 (3) By Remark \ref{rmk:de:6.7}, a BiHom-pre-Lie-commutative algebra $(L, \cdot, \ast, \a, \b)$ satisfying Eq.(\ref{eq:6.4}) is a BiHom-pre-Lie Poisson algebra;

 (4) A differential BiHom-Novikov-Poisson algebra is a BiHom-pre-Lie-commutative algebra satisfying Eqs.(\ref{eq:6.2}) and (\ref{eq:6.4}), hence is a BiHom-pre-Lie Poisson algebra.
 \end{rmk}

 \begin{cor}\label{cor:6.10} Let $(L, \cdot, \ast, \a, \b)$ be a regular BiHom-pre-Lie Poisson algebra. Then $(L, \cdot, [,], \a, \b)$ is a TBP algebra, where the operation $[,]$ is defined by Eq.(\ref{eq:6.5}).
 \end{cor}

 \begin{proof} From the proof of \cite[Proposition 3.4]{LMMP6} and Theorem \ref{thm:6.2}, we find that Eq.(\ref{eq:6.2}) is not used. The rest is obvious.
 \end{proof}

 By \cite[Theorem 2.1]{LMMP2}, we have:
 \begin{pro}\label{pro:6.11} Suppose that $(A, \cdot, \ast, \a_{A}, \b_{A})$ and $(B, \cdot, \ast, \a_{B}, \b_{B})$ are two BiHom-pre-Lie Poisson algebras. Define two bilinear operations $\cdot$ and $\ast$ on $A\o B$ by Eq.(\ref{eq:5.12}) and
 \begin{eqnarray}
 &(a\o b)\ast(a'\o b')=a\ast a'\o b\cdot b'+a\cdot a'\o b\ast b',&\label{eq:6.7}
 \end{eqnarray}
 respectively, where $a, a'\in A,\ b, b'\in B$. Then $(A\o B, \cdot, \ast, \a_{A}\o\a_{B}, \b_{A}\o\b_{B})$ is a BiHom-pre-Lie Poisson algebra.
 \end{pro}

\section{TBP 3-Lie algebra and 3-BiHom-Lie algebras}\label{se:tbp3} In this section, we discuss the properties of (T)BP 3-Lie algebras.

\subsection{3-BiHom-Lie algebras from TBP algebras}

 \begin{lem}\label{lem:5.4-1} Let $(L, \cdot, [,], \a, \b)$ be a TBP algebra. Then the following identities hold:
 \begin{eqnarray}
 &\hspace{-40mm}[\a^{p}\b^{q+2}(y)\cdot\a^{\ell+1}\b^{s+2}(v),
 \a^{m+2}\b^{n}(u)\cdot\a^{k+1}\b^{t+1}(z)]&\nonumber\\
 &\hspace{-10mm}+[\a^{m+1}\b^{n+1}(u)\cdot\a^{p+1}\b^{q+1}(y),
 \a^{k}\b^{t+2}(z)\cdot\a^{\ell+2}\b^{s+1}(v)]&\nonumber\\
 &\hspace{20mm}=2(\a^{m+1}\b^{n+1}(u)\cdot\a^{\ell+1}\b^{s+2}(v))\cdot[\a^{p+1}\b^{q+1}(y),
 \a^{k+1}\b^{t+1}(z)];& \label{eq:5.7}\\
 &\hspace{-40mm}\a^{p+2}\b^{q+2}(x)\cdot[\a^{\ell+1}\b^{s+2}(u),
 \a^{m+2}\b^{n}(y)\cdot\a^{k+2}\b^{t}(v)]&\nonumber\\
 &\hspace{-10mm}+\a^{k+1}\b^{t+3}(v)\cdot[\a^{p+1}\b^{q+1}(x)\cdot\a^{m+2}\b^{n}(y),
 \a^{\ell+2}\b^{s+1}(u)]&\nonumber\\
 &\hspace{20mm}+(\a^{m+1}\b^{n+2}(y)\cdot\a^{\ell+1}\b^{s+2}(u))\cdot[\a^{k+1}\b^{t+2}(v),
 \a^{p+3}\b^{q}(x)]=0,&\label{eq:5.8}
 \end{eqnarray}
 where $m, n, \ell, s, p, q, k, t\in \mathbb{Z}\  (\hbox{the set of integers}), x, y, z, u, v\in L$.
 \end{lem}

 \begin{proof} Eq.(\ref{eq:5.7}) can be obtained by Eq.(\ref{eq:5.1}). Next we prove that Eq.(\ref{eq:5.8}) holds. For all $m, n, \ell, s$, $p, q, k, t\in \mathbb{Z}, x, y, z, u, v\in L$, we have
 \begin{eqnarray*}
 &&\hspace{-20mm}2\a^{p+2}\b^{q+2}(x)\cdot[\a^{\ell+1}\b^{s+2}(u),\a^{m+2}\b^{n}(y)\cdot\a^{k+2}\b^{t}(v)] +2\a^{k+1}\b^{t+3}(v)\cdot[\a^{p+1}\b^{q+1}(x)\cdot\a^{m+2}\b^{n}(y),\\
 &&\a^{\ell+2}\b^{s+1}(u)]+2(\a^{m+1}\b^{n+2}(y)\cdot\a^{\ell+1}\b^{s+2}(u))\cdot[\a^{k+1}\b^{t+2}(v),\a^{p+3}\b^{q}(x)]\\
 &\stackrel{(\ref{eq:4.4})}=&\hspace{-5mm}2\a^{p+2}\b^{q+2}(x)\cdot[\a^{\ell+1}\b^{s+2}(u),
 \a^{m+2}\b^{n}(y)\cdot\a^{k+2}\b^{t}(v)]+2\a^{k+1}\b^{t+3}(v)\cdot[\a^{p+1}\b^{q+1}(x)\\
 &&\hspace{-5mm}\cdot\a^{m+2}\b^{n}(y),\a^{\ell+2}\b^{s+1}(u)]-2(\a^{m+1}\b^{n+2}(y)\cdot\a^{\ell+1}\b^{s+2}(u))\cdot[\a^{p+2}\b^{q+1}(x),\a^{k+2}\b^{t+1}(v)]\\
 &\stackrel{(\ref{eq:5.1})(\ref{eq:5.7})}=&\hspace{-5mm}[\a^{p+1}\b^{q+2}(x)
 \cdot\a^{\ell+1}\b^{s+2}(u),\a^{m+2}\b^{n+1}(y)\cdot\a^{k+2}\b^{t+1}(v)]+[\a^{\ell+1}\b^{s+3}(u),\a^{p+2}\b^{q+1}(x)\\
 &&\hspace{-5mm}\cdot(\a^{m+2}\b^{n}(y)\cdot\a^{k+2}\b^{t}(v))]+[\a^{k}\b^{t+3}(v)\cdot(\a^{p+1}\b^{q+1}(x)\cdot\a^{m+2}\b^{n}(y)),\a^{\ell+2}\b^{s+2}(u)]\\
 &&\hspace{-5mm}+[\a^{p+1}\b^{q+2}(x)\cdot\a^{m+2}\b^{n+1}(y),\a^{k+1}\b^{t+2}(v)
 \cdot\a^{\ell+2}\b^{s+1}(u)]-[\a^{m+1}\b^{n+2}(y)\cdot\a^{p+2}\b^{q+1}(x),\\
 &&\hspace{-5mm}\a^{k+1}\b^{t+2}(v)\cdot\a^{\ell+2}\b^{s+1}(u)]
 -[\a^{p+1}\b^{q+2}(x)\cdot\a^{\ell+2}\b^{s+2}(u),\a^{m+2}\b^{n+1}(y)\cdot\a^{k+2}\b^{t+1}(v)]\\
 &\stackrel{(\ref{eq:4.1})(\ref{eq:4.2})(\ref{eq:4.4})}=&\hspace{-2mm}0,
 \end{eqnarray*}
 finishing the proof.
 \end{proof}

 \begin{rmk} Let $p=-2, q=0, m=-2, n=0, \ell=-1, s=-1, k=-1, t=-1$ in Eq.(\ref{eq:5.8}). Then one has
 \begin{eqnarray}
 &\hspace{-4mm}\b^{2}(x)\cdot[\b(u),y\cdot\a\b^{-1}(v)]+\b^{2}(v)\cdot[\a^{-1}\b(x)\cdot y,\a(u)]+(\a^{-1}\b^{2}(y)\cdot\b(u))\cdot[\b(v),\a(x)]=0.&\label{eq:7.17}
 \end{eqnarray}
 \end{rmk}

 \begin{defi}\cite{KMS}\label{de:7.1} A {\bf 3-BiHom-Lie algebra} is a 4-tuple $(L, [\ ,\ ,\ ], \a, \b)$, where $L$ is a vector space, $\a,\b:L\lr L$ are two commuting linear maps and $[\ ,\ ,\ ]:L\times L\times L\lr L$ is a 3-linear map such that for all $x, y, z, u, v\in L$,{\small
 \begin{eqnarray}
 &\a([x,y,z])=[\a(x),\a(y),\a(z)],\ \  \b([x,y,z])=[\b(x),\b(y),\b(z)],&\label{eq:7.2}\\
 &[\b(x),\b(y),\a(z)]=-[\b(y),\b(x),\a(z)]=
 -[\b(x),\b(z),\a(y)],~\hbox{(BiHom-skew symmetry)}&\label{eq:7.3}\\
 &\hspace{-3mm}[\b^{2}(u),\b^{2}(v),[\b(x),\b(y),\a(z)]]=[\b^{2}(y),\b^{2}(z),[\b(u),\b(v),\a(x)]]
 -[\b^{2}(x),\b^{2}(z),[\b(u),\b(v),\a(y)]]&\nonumber\\
 &\hspace{40mm}+[\b^{2}(x),\b^{2}(y),[\b(u),\b(v),\a(z)]].\hbox{(3-BiHom-Jacobi identity)}& \label{eq:7.4}
 \end{eqnarray}}
 \end{defi}

 \begin{thm}\label{thm:7.8} Let $(L, \cdot, [,], \a, \b)$ be a regular TBP algebra and $D$ be a derivation of $(L,\cdot,\a,\b)$ and $(L, [,], \a, \b)$. Define a ternary operation on $L$ by
 \begin{eqnarray}
 &\hspace{-6mm}[x,y,z]=D(x)\cdot[\b^{-1}(y),\b^{-1}(z)]+D(y)\cdot[\a^{-1}(z),\a\b^{-2}(x)]
 +D\a^{-1}\b(z)\cdot[\b^{-1}(x),\a\b^{-2}(y)].&\label{eq:7.8}
 \end{eqnarray}
 Then $(L, [\ ,\ ,\ ], \a, \b)$ is a 3-BiHom-Lie algebra.
 \end{thm}

 \begin{proof} We only check Eq.(\ref{eq:7.4}) below, and others are direct. Firstly, for all $x, y, z, u, v\in L$, one has {\small
 \begin{eqnarray*}
 &&\hspace{-15mm}2D\a^{-1}\b^{2}(x)\cdot D([\a^{-1}\b(y),z])+2D\a^{-1}\b^{2}(y)\cdot D([\a^{-1}\b(z),x])+2D\a^{-1}\b^{2}(z)\cdot D([\a^{-1}\b(x),y])\\
 &&\hspace{-11mm}\stackrel{(\ref{eq:5.1})}=[D\a^{-2}\b^{2}(x)\cdot D\a^{-1}\b(y),\b(z)]+[D\a^{-1}\b^{2}(y),D\a^{-1}\b(x)\cdot z]+[D\a^{-2}\b^{2}(x)\cdot\a^{-1}\b(y),D\b(z)]\\
 &&\hspace{-7mm}+[\a^{-1}\b^{2}(y),D\a^{-1}\b(x)\cdot D(z)]+[D\a^{-2}\b^{2}(y)\cdot D\a^{-1}\b(z),\b(x)]+[D\a^{-1}\b^{2}(z),D\a^{-1}\b(y)\cdot x]\\
 &&\hspace{-7mm}+[D\a^{-2}\b^{2}(y)\cdot\a^{-1}\b(z),D\b(x)]+[\a^{-1}\b^{2}(z),D\a^{-1}\b(y)\cdot D(x)]+[D\a^{-2}\b^{2}(z)\cdot D\a^{-1}\b(x),\b(y)]\\
 &&\hspace{-7mm}+[D\a^{-1}\b^{2}(x),D\a^{-1}\b(z)\cdot y]+[D\a^{-2}\b^{2}(z)\cdot\a^{-1}\b(x),D\b(y)]+[\a^{-1}\b^{2}(x),D\a^{-1}\b(z)\cdot D(y)]\\
 &&\hspace{-12mm}\stackrel{(\ref{eq:4.2})(\ref{eq:4.4})}=[D\a^{-1}\b^{2}(y),D\a^{-1}\b(x)\cdot z]+[D\a^{-2}\b^{2}(x)\cdot\a^{-1}\b(y),D\b(z)]+[D\a^{-1}\b^{2}(z),D\a^{-1}\b(y)\cdot x]\\
 &&\hspace{-7mm}+[D\a^{-2}\b^{2}(y)\cdot\a^{-1}\b(z),D\b(x)]
 +[D\a^{-1}\b^{2}(x),D\a^{-1}\b(z)\cdot y]+[D\a^{-2}\b^{2}(z)\cdot\a^{-1}\b(x),D\b(y)]\\
 &&\hspace{-13mm}\stackrel{(\ref{eq:4.2})(\ref{eq:4.4})(\ref{eq:5.1})}=\hspace{-2mm}
 -2\a^{-1}\b^{2}(x)\cdot[D\a^{-1}\b(y),D(z)]-2\a^{-1}\b^{2}(y)\cdot[D\a^{-1}\b(z),D(x)] \\
 &&\hspace{-3mm}-2\a^{-1}\b^{2}(z)\cdot[D\a^{-1}\b(x),D(y)].
 \end{eqnarray*}}
 So{\small
 \begin{eqnarray}
 &\hspace{-15mm}D\a^{-1}\b^{2}(x)\cdot D([\a^{-1}\b(y),z])+D\a^{-1}\b^{2}(y)\cdot D([\a^{-1}\b(z),x])+D\a^{-1}\b^{2}(z)\cdot D([\a^{-1}\b(x),y])&\nonumber\\
 &\hspace{-10mm}=-\a^{-1}\b^{2}(x)\cdot[D\a^{-1}\b(y),D(z)]-\a^{-1}\b^{2}(y)\cdot[D\a^{-1}\b(z),D(x)]
 -\a^{-1}\b^{2}(z)\cdot[D\a^{-1}\b(x),D(y)].&\label{eq:7.9}
 \end{eqnarray} }
 Secondly, we calculate {\small
 \begin{eqnarray*}
 &&\hspace{-19mm}[\b^{2}(y),\b^{2}(z),[\b(u),\b(v),\a(x)]]
 -[\b^{2}(x),\b^{2}(z),[\b(u),\b(v),\a(y)]]
 +[\b^{2}(x),\b^{2}(y),[\b(u),\b(v),\a(z)]]\\
 &\stackrel{(\ref{eq:7.9})}=&\hspace{-3mm}\Big((D^{2}\a^{-1}\b^{2}(x)
 \cdot[\a^{-1}\b(u),v])\cdot[\b(y),\a(z)]
 -(D^{2}\a^{-1}\b^{2}(y)\cdot[\a^{-1}\b(u),v])\cdot[\b(x),\a(z)]\\
 &&\hspace{-6mm}+(D^{2}\a^{-1}\b^{2}(z)\cdot[\a^{-1}\b(u),v])\cdot[\b(x),\a(y)]\Big)
 +\Big((D^{2}\a^{-1}\b^{2}(u)[\a^{-1}\b(v),x])\cdot[\b(y),\a(z)]\\
 &&\hspace{-6mm}-(D^{2}\a^{-1}\b^{2}(u)\cdot[\a^{-1}\b(v),y])\cdot[\b(x),\a(z)]
 +(D^{2}\a^{-1}\b^{2}(u)\cdot[\a^{-1}\b(v),z])\cdot[\b(x),\a(y)]\\
 &&\hspace{-6mm}+(D^{2}\a^{-1}\b^{2}(v)\cdot[\a^{-1}\b(x),u])\cdot[\b(y),\a(z)]
 -(D^{2}\a^{-1}\b^{2}(v)\cdot[\a^{-1}\b(y),u])\cdot[\b(x),\a(z)]\\
 &&\hspace{-6mm}+(D^{2}\a^{-1}\b^{2}(v)[\a^{-1}\b(z),u])[\b(x),\a(y)]\Big)
 +\Big(\hspace{-1.5mm}-(\a^{-1}\b^{2}(x)\cdot[D\a^{-1}\b(u),D(v)])\cdot[\b(y),\a(z)]\\
 &&\hspace{-6mm}+(\a^{-1}\b^{2}(y)\cdot[D\a^{-1}\b(u),D(v)])\cdot[\b(x),\a(z)]
 -(\a^{-1}\b^{2}(z)\cdot[D\a^{-1}\b(u),D(v)])\cdot[\b(x),\a(y)]\Big)\\
 &&\hspace{-8mm}+\Big(\hspace{-1.5mm}-(\a^{-1}\b^{2}(u)\cdot[D\a^{-1}\b(v),D(x)])\cdot[\b(y),\a(z)]+(\a^{-1}\b^{2}(u)\cdot[D\a^{-1}\b(v),D(y)])\cdot[\b(x),\a(z)]\\
 &&\hspace{-8mm}-(\a^{-1}\b^{2}(u)\cdot[D\a^{-1}\b(v),D(z)])\cdot[\b(x),\a(y)]\Big)
 +\Big(\hspace{-1.5mm}-(\a^{-1}\b^{2}(v)\cdot[D\a^{-1}\b(x),D(u)])\cdot[\b(y),\a(z)]\\
 &&\hspace{-6mm}+(\a^{-1}\b^{2}(v)\cdot[D\a^{-1}\b(y),D(u)])\cdot[\b(x),\a(z)]
 -(\a^{-1}\b^{2}(v)\cdot[D\a^{-1}\b(z),D(u)])\cdot[\b(x),\a(y)]\Big)\\
 &&\hspace{-6mm}+\Big(\hspace{-1.5mm}-D\b^{2}(x)\cdot[\b(z),D(u)\cdot[\b^{-1}(v),\a\b^{-2}(y)]]
 +D\b^{2}(x)\cdot[\b(y),D(u)\cdot[\b^{-1}(v),\a\b^{-2}(z)]]\\
 &&\hspace{-6mm}+D\b^{2}(y)\cdot[\b(z),D(u)\cdot[\b^{-1}(v),\a\b^{-2}(x)]]
 +D\b^{2}(y)\cdot[D\a^{-1}\b(u)\cdot[\a^{-1}(v),\b^{-1}(z)],\a(x)]\\
 &&\hspace{-6mm}+D\b^{2}(z)\cdot[D\a^{-1}\b(u)\cdot[\a^{-1}(v),\b^{-1}(x)],\a(y)]
 -D\b^{2}(z)\cdot[D\a^{-1}\b(u)\cdot[\a^{-1}(v),\b^{-1}(y)],\a(x)]\Big)\\
 &&\hspace{-6mm}+\Big(\hspace{-1.5mm}-D\b^{2}(x)\cdot[\b(z),D(v)\cdot[\b^{-1}(y),\a\b^{-2}(u)]]
 +D\b^{2}(x)\cdot[\b(y),D(v)\cdot[\b^{-1}(z),\a\b^{-2}(u)]]\\
 &&\hspace{-6mm}+D\b^{2}(y)\cdot[\b(z),D(v)\cdot[\b^{-1}(x),\a\b^{-2}(u)]]
 +D\b^{2}(y)\cdot[D\a^{-1}\b(v)\cdot[\a^{-1}(z),\b^{-1}(u)],\a(x)]\\
 &&\hspace{-6mm}+D\b^{2}(z)\cdot[D\a^{-1}\b(v)\cdot[\a^{-1}(x),\b^{-1}(u)],\a(y)]
 -D\b^{2}(z)\cdot[D\a^{-1}\b(v)\cdot[\a^{-1}(y),\b^{-1}(u)],\a(x)]\Big)\\
 &&\hspace{-6mm}+\Big(D\b^{2}(x)\cdot[\b(y),D(z)\cdot[\b^{-1}(u),\a\b^{-2}(v)]]
 +D\b^{2}(z)\cdot[D\a^{-1}\b(x)\cdot[\a^{-1}(u),\b^{-1}(v)],\a(y)]\\
 &&\hspace{-6mm}+D\b^{2}(y)\cdot[\b(z),D(x)\cdot[\b^{-1}(u),\a\b^{-2}(v)]]
 -D\b^{2}(x)\cdot[\b(z),D(y)\cdot[\b^{-1}(u),\a\b^{-2}(v)]]\\
 &&\hspace{-6mm}-D\b^{2}(z)\cdot[D\a^{-1}\b(y)\cdot[\a^{-1}(u),\b^{-1}(v)],\a(x)]
 +D\b^{2}(y)\cdot[D\a^{-1}\b(z)\cdot[\a^{-1}(u),\b^{-1}(v)],\a(x)]\Big)\\
 &\stackrel{\bigtriangleup}=&\hspace{-4mm}\hbox{(I)}+\hbox{(II)}+\hbox{(III)}+\hbox{(IV)}+\hbox{(V)}+\hbox{(VI)}+\hbox{(VII)}+\hbox{(VIII)}.
 \end{eqnarray*}}
 While {\small
 \begin{eqnarray*}
 \hbox{(I)}\hspace{-7mm} &\stackrel{(\ref{eq:4.1})(\ref{eq:4.2})(\ref{eq:4.4})}=&\hspace{-7mm}(D^{2}\a^{-1}\b^{2}(x)\cdot[\a^{-1}\b(y),z])\cdot[\b(u),\a(v)]+(D^{2}\a^{-1}\b^{2}(y)\cdot[\a^{-1}\b(z),x])\cdot[\b(u),\a(v)]\\
 &&+(D^{2}\a^{-1}\b^{2}(z)\cdot[\a^{-1}\b(x),y])\cdot[\b(u),\a(v)].\\
 \hbox{(II)}&\stackrel{(\ref{eq:4.1})(\ref{eq:4.2})(\ref{eq:4.4})(\ref{eq:5.6})}=&0.\\
 \hbox{(III)}&\stackrel{(\ref{eq:4.1})(\ref{eq:4.2})(\ref{eq:4.4})(\ref{eq:5.3})}=&0.
 \end{eqnarray*}}
 Further, {\small
 \begin{eqnarray*}
 &&\hspace{-6mm}-(\a^{-1}\b^{2}(u)\cdot[D\a^{-1}\b(v),D(x)])\cdot[\b(y),\a(z)]-D\b^{2}(x)\cdot[\b(z),D(v)\cdot[\b^{-1}(y),\a\b^{-2}(u)]]\\
 &&\hspace{90mm}+D\b^{2}(x)\cdot[\b(y),D(v)\cdot[\b^{-1}(z),\a\b^{-2}(u)]]\\
 &&\hspace{-6mm}\stackrel{(\ref{eq:4.1})(\ref{eq:4.2})(\ref{eq:4.4})(\ref{eq:5.4})}=-([\a^{-2}\b^{2}(z),\a^{-1}\b(y)]\cdot\b(u))\cdot[D\b(x),D\a(v)]-D\b^{2}(x)\cdot[D\a^{-1}\b(v)\cdot[\a^{-1}(z),\b^{-1}(y)],\a(u)]\\
 &&\stackrel{(\ref{eq:7.17})(\ref{eq:4.4})}= -D\b^{2}(v)\cdot[\b(u),D(x)\cdot[\b^{-1}(y),\a\b^{-2}(z)]].
 \end{eqnarray*}}
 Based on the above identity, we obtain {\small
 \begin{eqnarray*}
 \hbox{(IV)}+\hbox{(VII)}\hspace{-4mm}&=&\hspace{-4mm}-D\b^{2}(v)\cdot[\b(u),D(x)\cdot[\b^{-1}(y),\a\b^{-2}(z)]]-D\b^{2}(v)\cdot[\b(u),D(y)\cdot[\b^{-1}(z),\a\b^{-2}(x)]]\\
 &&-D\b^{2}(v)\cdot[\b(u),D(z)\cdot[\b^{-1}(x),\a\b^{-2}(y)]]\\
 &\stackrel{(\ref{eq:4.4})}=&\hspace{-3mm}D\b^{2}(v)\cdot[D\a^{-1}\b(x)\cdot[\a^{-1}(y),\b^{-1}(z)],\a(u)]+D\b^{2}(v)\cdot[D\a^{-1}\b(y)\cdot[\a^{-1}(z),\b^{-1}(x)],\a(u)]\\
 &&+D\b^{2}(v)\cdot[D\a^{-1}\b(z)\cdot[\a^{-1}(x),\b^{-1}(y)],\a(u)],\\
 \hbox{(V)}+\hbox{(VI)}\hspace{-4mm}&=&\hspace{-4mm}D\b^{2}(u)\cdot[\b(v),D(x)\cdot[\b^{-1}(y),\a\b^{-2}(z)]]+D\b^{2}(u)\cdot[\b(v),D(y)\cdot[\b^{-1}(z),\a\b^{-2}(x)]]\\
 &&\hspace{-4mm}+D\b^{2}(u)\cdot[\b(v),D(z)\cdot[\b^{-1}(x),\a\b^{-2}(y)]].
 \end{eqnarray*}}
 From Eq.(\ref{eq:7.17}), we can get {\small
 \begin{eqnarray*}
 \hbox{(VIII)}&=&-(\a^{-1}\b^{2}(x)\cdot[D\a^{-1}\b(y),D(z)])\cdot[\b(u),\a(v)]-(\a^{-1}\b^{2}(y)\cdot[D\a^{-1}\b(z),D(x)])\cdot[\b(u),\a(v)]\\
 &&-(\a^{-1}\b^{2}(z)\cdot[D\a^{-1}\b(x),D(y)])\cdot[\b(u),\a(v)].
 \end{eqnarray*}}
 In summary, Eq.(\ref{eq:7.4}) holds.
 \end{proof}

 Another approach to construct 3-BiHom-Lie algebras from TBP algebras is provided below.

 \begin{thm}\label{thm:7.12} Let $(L, \cdot, [,], \a, \b)$ be a regular TBP algebra and $f: L\lr L$ a linear map such that $f^{2}=\id$,
 \begin{eqnarray}
 &f([x,y])=-[f(x),f(y)],\ \forall x,y\in L, &\label{eq:7.11}
 \end{eqnarray}
 and any two of the maps $\a, \b, f$ commute.  Define a ternary operation on $L$ by
 \begin{eqnarray}
 &\hspace{-8mm}[x,y,z]=f(x)\cdot[\b^{-1}(y),\b^{-1}(z)]+f(y)\cdot[\a^{-1}(z),\a\b^{-2}(x)] +f(\a^{-1}\b(z))\cdot[\b^{-1}(x),\a\b^{-2}(y)].&\label{eq:7.12}
 \end{eqnarray}
 Then $(L, [\ ,\ ,\ ], \a, \b)$ is a 3-BiHom-Lie algebra.
 \end{thm}

 \begin{proof} We only do the following calculations and others are direct. For all $x, y, z, u, v\in L$, by Eqs.(\ref{eq:7.11}) and (\ref{eq:7.12}), one gets {\small
 \begin{eqnarray*}
 &&\hspace{-15mm}[\b^{2}(y),\b^{2}(z),[\b(u),\b(v),\a(x)]]
 -[\b^{2}(x),\b^{2}(z),[\b(u),\b(v),\a(y)]]+[\b^{2}(x),\b^{2}(y),[\b(u),\b(v),\a(z)]]\\
 &&-[\b^{2}(u),\b^{2}(v),[\b(x),\b(y),\a(z)]]\\
 &=&\hspace{-3mm}\Big(\hspace{-1.5mm}-(\a^{-1}\b^{2}(x)\cdot[f\a^{-1}\b(u),f(v)])
 \cdot[\b(y),\a(z)]+
 (\a^{-1}\b^{2}(y)\cdot[f\a^{-1}\b(u),f(v)])\cdot[\b(x),\a(z)]\\
 &&\hspace{-3mm}-(\a^{-1}\b^{2}(z)\cdot[f\a^{-1}\b(u),f(v)])\cdot[\b(x),\a(y)]\Big)
 +\Big(\hspace{-1.5mm}-f\b^{2}(z)\cdot[f\a^{-1}\b(y)\cdot[\a^{-1}(u),\b^{-1}(v)],\a(x)]\\
 &&\hspace{-3mm}+f\b^{2}(y)\cdot[f\a^{-1}\b(z)\cdot[\a^{-1}(u),\b^{-1}(v)],\a(x)] +(\a^{-1}\b^{2}(x)\cdot[f\a^{-1}\b(y),f(z)])\cdot[\b(u),\a(v)]\Big)\\
 &&\hspace{-3mm}+\Big(f\b^{2}(x)\cdot[\b(y),f(z)\cdot[\b^{-1}(u),\a\b^{-2}(v)]]
 +(\a^{-1}\b^{2}(y)\cdot[f\a^{-1}\b(z),f(x)])\cdot[\b(u),\a(v)]\\
 &&\hspace{-3mm}+f\b^{2}(z)\cdot[f\a^{-1}\b(x)\cdot[\a^{-1}(u),\b^{-1}(v)],\a(y)]\Big)
 +\Big(f\b^{2}(y)\cdot[\b(z),f(x)\cdot[\b^{-1}(u),\a\b^{-2}(v)]]\\
 &&\hspace{-3mm}-f\b^{2}(x)\cdot[\b(z),f(y)\cdot[\b^{-1}(u),\a\b^{-2}(v)]] +(\a^{-1}\b^{2}(z)\cdot[f\a^{-1}\b(x),f(y)])\cdot[\b(u),\a(v)]\Big)\\
 &&\hspace{-3mm}+\Big(\hspace{-1.5mm}-(\a^{-1}\b^{2}(v)\cdot[f\a^{-1}\b(x),f(u)])\cdot[\b(y),\a(z)]
 -f\b^{2}(x)\cdot[\b(z),f(u)\cdot[\b^{-1}(v),\a\b^{-2}(y)]]\\
 &&\hspace{-3mm}-f\b^{2}(u)\cdot[\b(v),f(x)\cdot[\b^{-1}(y),\a\b^{-2}(z)]]
 +f\b^{2}(x)\cdot[\b(y),f(u)\cdot[\b^{-1}(v),\a\b^{-2}(z)]]\Big)\\
 &&\hspace{-3mm}+\Big(f\b^{2}(y)\cdot[\b(z),f(u)\cdot[\b^{-1}(v),\a\b^{-2}(x)]]
 +(\a^{-1}\b^{2}(v)\cdot[f\a^{-1}\b(y),f(u)])\cdot[\b(x),\a(z)]\\
 &&\hspace{-3mm}+f\b^{2}(y)\cdot[f\a^{-1}\b(u)\cdot[\a^{-1}(v),\b^{-1}(z)],\a(x)]
 -f\b^{2}(u)\cdot[\b(v),f(y)\cdot[\b^{-1}(z),\a\b^{-2}(x)]]\Big)\\
 &&\hspace{-3mm}+\Big(f\b^{2}(z)\cdot[f\a^{-1}\b(u)\cdot[\a^{-1}(v),\b^{-1}(x)],\a(y)]
 -f\b^{2}(z)\cdot[f\a^{-1}\b(u)\cdot[\a^{-1}(v),\b^{-1}(y)],\a(x)]\\
 &&\hspace{-3mm}-(\a^{-1}\b^{2}(v)\cdot[f\a^{-1}\b(z),f(u)])\cdot[\b(x),\a(y)]
 -f\b^{2}(u)\cdot[\b(v),f(z)\cdot[\b^{-1}(x),\a\b^{-2}(y)]]\Big)\\
 &&\hspace{-3mm}+\Big(\hspace{-1.5mm}-(\a^{-1}\b^{2}(u)\cdot[f\a^{-1}\b(v),f(x)])\cdot[\b(y),\a(z)]
 -f\b^{2}(x)\cdot[\b(z),f(v)\cdot[\b^{-1}(y),\a\b^{-2}(u)]]\\
 &&\hspace{-3mm}+f\b^{2}(x)\cdot[\b(y),f(v)\cdot[\b^{-1}(z),\a\b^{-2}(u)]]
 -f\b^{2}(v)\cdot[f\a^{-1}\b(x)\cdot[\a^{-1}(y),\b^{-1}(z)],\a(u)]\Big)\\
 &&\hspace{-3mm}+\Big(f\b^{2}(y)\cdot[\b(z),f(v)\cdot[\b^{-1}(x),\a\b^{-2}(u)]]
 +(\a^{-1}\b^{2}(u)\cdot[f\a^{-1}\b(v),f(y)])\cdot[\b(x),\a(z)]\\
 &&\hspace{-3mm}+f\b^{2}(y)\cdot[f\a^{-1}\b(v)\cdot[\a^{-1}(z),\b^{-1}(u)],\a(x)]
 -f\b^{2}(v)\cdot[f\a^{-1}\b(y)\cdot[\a^{-1}(z),\b^{-1}(x)],\a(u)]\Big)\\
 &&\hspace{-3mm}+\Big(f\b^{2}(z)\cdot[f\a^{-1}\b(v)\cdot[\a^{-1}(x),\b^{-1}(u)],\a(y)]
 -f\b^{2}(z)\cdot[f\a^{-1}\b(v)\cdot[\a^{-1}(y),\b^{-1}(u)],\a(x)]\\
 &&\hspace{-3mm}-(\a^{-1}\b^{2}(u)\cdot[f\a^{-1}\b(v),f(z)])\cdot[\b(x),\a(y)]
 -f\b^{2}(v)\cdot[f\a^{-1}\b(z)\cdot[\a^{-1}(x),\b^{-1}(y)],\a(u)]\Big)\\
 &\stackrel{\bigtriangleup}=&\hbox{(I)}+\hbox{(II)}+\hbox{(III)}+\hbox{(IV)}
 +\hbox{(V)}+\hbox{(VI)}+\hbox{(VII)}+\hbox{(VIII)}+\hbox{(IX)}+\hbox{(X)}.
 \end{eqnarray*}}
 \vskip-5mm
 While {\small
 \begin{eqnarray*}
 \hbox{(I)}\hspace{-8mm}&=&\hspace{-8mm}-(\a^{-1}\b^{2}(x)\cdot[f\a^{-1}\b(u),f(v)])
 \cdot[\b(y),\a(z)]+(\a^{-1}\b^{2}(y)\cdot[f\a^{-1}\b(u),f(v)])\cdot[\b(x),\a(z)]\\
 &&-(\a^{-1}\b^{2}(z)\cdot[f\a^{-1}\b(u),f(v)])\cdot[\b(x),\a(y)]\\
 &\stackrel{(\ref{eq:4.1})(\ref{eq:4.2})(\ref{eq:4.4})(\ref{eq:5.3})}=&0.
 \end{eqnarray*}}
 \vskip-5mm
 One can calculate as follows: {\small
 \begin{eqnarray*}
 &&\hspace{-18mm}(\a^{-1}\b^{2}(x)\cdot[f\a^{-1}\b(y),f(z)])\cdot[\b(u),\a(v)]\\
 &=&2(\a^{-1}\b^{2}(x)\cdot[f\a^{-1}\b(y),f(z)])\cdot[\b(u),\a(v)]-
 (\a^{-1}\b^{2}(x)\cdot[f\a^{-1}\b(y),f(z)])\cdot[\b(u),\a(v)]\\
 &\stackrel{(\ref{eq:4.1})(\ref{eq:4.2})}=&2\b^{2}(x)\cdot([\a^{-1}\b(u),v]\cdot[f(y),f\a\b^{-1}(z)])
 -(\a^{-1}\b^{2}(x)\cdot[f\a^{-1}\b(y),f(z)])\cdot[\b(u),\a(v)]\\
 &\stackrel{(\ref{eq:5.1})(\ref{eq:5.3})}=&\b^{2}(x)\cdot[[\a^{-2}\b(u),\a^{-1}(v)]\cdot f(y),f\a(z)]+\b^{2}(x)\cdot[f\b(y),[\a^{-1}(u),\b^{-1}(v)]\cdot f\a\b^{-1}(z)]\\
 &&+(f\a^{-1}\b^{2}(y)\cdot[f\a^{-1}\b(z),x])\cdot[\b(u),\a(v)]
 +(f\a^{-1}\b^{2}(z)\cdot[\a^{-1}\b(x),f(y)])\cdot[\b(u),\a(v)],
 \end{eqnarray*}}
 then based on the above equation, we have {\small
 \begin{eqnarray*}
 \hbox{(II)}\hspace{-6mm}&=&\hspace{-6mm}-f\b^{2}(z)\cdot[f\a^{-1}\b(y)\cdot[\a^{-1}(u),\b^{-1}(v)],\a(x)]+\b^{2}(x)\cdot[[\a^{-2}\b(u),\a^{-1}(v)]\cdot f(y),f\a(z)]\\
 &&\hspace{-6mm}+\b^{2}(x)\cdot[f\b(y),[\a^{-1}(u),\b^{-1}(v)]\cdot f\a\b^{-1}(z)]+f\b^{2}(y)\cdot[f\a^{-1}\b(z)\cdot[\a^{-1}(u),\b^{-1}(v)],\a(x)]\\
 &&\hspace{-6mm}+(f\a^{-1}\b^{2}(y)\cdot[f\a^{-1}\b(z),x])\cdot[\b(u),\a(v)]+(f\a^{-1}\b^{2}(z)\cdot[\a^{-1}\b(x),f(y)])\cdot[\b(u),\a(v)]\\
 &\stackrel{(\ref{eq:5.3})}=&\hspace{-3mm}-(f\a^{-1}\b^{2}(z)\cdot[\a^{-1}\b(u),v])\cdot[\b(x),f\a(y)]-(f\a^{-1}\b^{2}(y)\cdot[\a^{-1}\b(u),v])\cdot[f\b(z),\a(x)]\\
 &&+(f\a^{-1}\b^{2}(y)\cdot[f\a^{-1}\b(z),x])\cdot[\b(u),\a(v)]+(f\a^{-1}\b^{2}(z)\cdot[\a^{-1}\b(x),f(y)])\cdot[\b(u),\a(v)]\\
 &\stackrel{(\ref{eq:4.1})(\ref{eq:4.2})}=&\hspace{-2mm}-f\b^{2}(z)\cdot([\a^{-1}\b(u),v]\cdot[x,f\a\b^{-1}(y)])-f\b^{2}(y)\cdot([\a^{-1}\b(u),v]\cdot[f(z),\a\b^{-1}(x)])\\
 &&\hspace{-6mm}+f\b^{2}(y)\cdot([\a^{-1}\b(u),v]\cdot[f(z),\a\b^{-1}(x)])+f\b^{2}(z)\cdot([\a^{-1}\b(u),v]\cdot[x,f\a\b^{-1}(y)])\\
 &=&\hspace{-6mm}0.
 \end{eqnarray*}}
 Similarly, $\hbox{(III)=(IV)}=0$.

 Furthermore, one can obtain {\small
 \begin{eqnarray*}
 \hbox{(V)}\hspace{-2mm}&=&\hspace{-2mm}-(\a^{-1}\b^{2}(v)\cdot[f\a^{-1}\b(x),f(u)])\cdot[\b(y),\a(z)]-f\b^{2}(x)\cdot[\b(z),f(u)\cdot[\b^{-1}(v),\a\b^{-2}(y)]]\\
 &&-f\b^{2}(u)\cdot[\b(v),f(x)\cdot[\b^{-1}(y),\a\b^{-2}(z)]]+f\b^{2}(x)\cdot[\b(y),f(u)\cdot[\b^{-1}(v),\a\b^{-2}(z)]]\\
 &\stackrel{(\ref{eq:5.4})}=&f\b^{2}(u)\cdot[f\a^{-1}\b(x)\cdot[\a^{-1}(y),\b^{-1}(z)],\a(v)]-(\a^{-1}\b^{2}(v)\cdot[f\a^{-1}\b(x),f(u)])\cdot[\b(y),\a(z)]\\
 &&-f\b^{2}(x)\cdot[f\a^{-1}\b(u)\cdot[\a^{-1}(y),\b^{-1}(z)],\a(v)].
 \end{eqnarray*}}
 Therefore \hbox{(V)}=0 since \hbox{(II)}=0. Likely, \hbox{(VI)=(VII)=(VIII)=(IX)=(X)}=0. Thus, Eq.(\ref{eq:7.4}) holds. The proof is completed.
 \end{proof}

\subsection{Strong BP 3-Lie algebras from strong BP algebras}

 \begin{defi}\label{de:7.4} A {\bf BP 3-Lie algebra} is a 5-tuple $(L, \cdot, [\ ,\ ,\ ], \a, \b)$, where $(L, \cdot, \a, \b)$ is a BiHom-commutative algebra and $(L, [\ ,\ ,\ ], \a, \b)$ is a 3-BiHom-Lie algebra such that
 \begin{eqnarray}
 &[\a\b(x),\a\b(y),u\cdot v]=[\b(x),\b(y),u]\cdot\b(v)+\b(u)\cdot[\a(x),\a(y),v],\ \forall x,y,u,v\in L.&\label{eq:7.6}
 \end{eqnarray}

 Furthermore, if for all $x, y, z, u, v, w \in L$,{\small
 \begin{eqnarray}
 &\hspace{-8mm}0=[\b^{2}(x),\b^{2}(y),\a\b(u)]\cdot[\a\b(z),\a\b(v),\a^{2}(w)]
 -[\b^{2}(x),\b^{2}(y),\a\b(z)]\cdot[\a\b(u),\a\b(v),\a^{2}(w)]&\nonumber\\
 &\hspace{-8mm}+[\b^{2}(x),\b^{2}(y),\a\b(w)]\cdot[\a\b(z),\a\b(u),\a^{2}(v)]
 -[\b^{2}(x),\b^{2}(y),\a\b(v)]\cdot[\a\b(z),\a\b(u),\a^{2}(w)],&\label{eq:7.7}
 \end{eqnarray}}
 then we call the BP 3-Lie algebra $(L, \cdot, [\ ,\ ,\ ], \a, \b)$ {\bf strong}.
 \end{defi}

 \begin{pro}\label{pro:7.5} Let $(L, \cdot, [\ ,\ ,\ ])$ be a Poisson 3-Lie algebra \cite{D} and $\a,\b:L\lr L$ two commuting linear maps such that $\a([x,y,z])=[\a(x),\a(y),\a(z)], \b([x,y,z])=[\b(x),\b(y),\b(z)]$. Then $(L, \cdot'=\cdot\circ(\a\o\b), [\ ,\ ,\ ]'=[\ ,\ ,\ ]\circ(\a\o\a\o\b), \a, \b)$ is a BP 3-Lie algebra, called the ``Yau twist" of $(L, \cdot, [\ ,\ ,\ ]).$
 \end{pro}

 \begin{proof} By \cite[Theorem 1.12]{KMS} and Remark \ref{rmk:de:4.1}, we only need to check that Eq.(\ref{eq:7.6}) holds. For all $x,y,u,v\in L$, one calculates
 \begin{eqnarray*}
 &&\hspace{-3mm}[\b(x),\b(y),u]'\cdot'\b(v)+\b(u)\cdot'[\a(x),\a(y),v]'\\
 &&=[\a^{2}\b(x),\a^{2}\b(y),\a\b(u)]\cdot\b^{2}(v)+\a\b(u)\cdot[\a^{2}\b(x),\a^{2}\b(y),\b^{2}(v)]\\
 &&=[\a^{2}\b(x),\a^{2}\b(y),\a\b(u)\cdot \b^{2}(v)]\\
 &&=[\a\b(x),\a\b(y),u\cdot' v]',
 \end{eqnarray*}
 as desired.
 \end{proof}

 \begin{thm}\label{thm:7.7} Let $(L, \cdot, [,], \a, \b)$ be a regular strong BP algebra and $D$ be a derivation of $(L, \cdot, \a, \b)$ and $(L, [,], \a, \b)$ such that any two of the maps $\a, \b, D$ commute. Define a ternary operation on $L$ by Eq.(\ref{eq:7.8}). Then $(L, \cdot, [\ ,\ ,\ ], \a, \b)$ is a strong BP 3-Lie algebra.
 \end{thm}

 \begin{proof} We take three steps to prove this result.

 {\bf Step 1.} $(L, [\ ,\ ,\ ], \a, \b)$ is a 3-BiHom-Lie algebra. In this step, we only verify Eq.(\ref{eq:7.4}) below, and others are direct. For all $x, y, z, u, v, w \in L$, we have {\small
 \begin{eqnarray*}
 &&\hspace{-9mm}[\b^{2}(u),\b^{2}(v),[\b(x),\b(y),\a(z)]]
 -[\b^{2}(y),\b^{2}(z),[\b(u),\b(v),\a(x)]]+[\b^{2}(x),\b^{2}(z),[\b(u),\b(v),\a(y)]]\\
 &&-[\b^{2}(x),\b^{2}(y),[\b(u),\b(v),\a(z)]]\\
 &=&\hspace{-2mm}\Big((D^{2}\a^{-1}\b^{2}(x)\cdot[\a^{-1}\b(y),z])\cdot[\b(u),\a(v)]
 +(D^{2}\a^{-1}\b^{2}(y)\cdot[\a^{-1}\b(z),x])\cdot[\b(u),\a(v)]+(D^{2}\a^{-1}\b^{2}(z)\\
 &&\hspace{-2mm}\cdot[\a^{-1}\b(x),y])\cdot[\b(u),\a(v)]
 -(D^{2}\a^{-1}\b^{2}(x)\cdot[\a^{-1}\b(u),v])\cdot[\b(y),\a(z)]
 +(D^{2}\a^{-1}\b^{2}(y)\cdot[\a^{-1}\b(u),v])\\
 &&\hspace{-2mm}\cdot[\b(x),\a(z)]
 -(D^{2}\a^{-1}\b^{2}(z)\cdot[\a^{-1}\b(u),v])\cdot[\b(x),\a(y)]\Big)
 +\Big(\hspace{-1.5mm}-(D^{2}\a^{-1}\b^{2}(u)\cdot[\a^{-1}\b(v),x])\cdot[\b(y),\a(z)]\\
 &&\hspace{-2mm}+(D^{2}\a^{-1}\b^{2}(u)\cdot[\a^{-1}\b(v),y])\cdot[\b(x),\a(z)] -(D^{2}\a^{-1}\b^{2}(u)\cdot[\a^{-1}\b(v),z])\cdot[\b(x),\a(y)]-(D^{2}\a^{-1}\b^{2}(v)\\
 &&\hspace{-2mm}\cdot[\a^{-1}\b(x),u])\cdot[\b(y),\a(z)]
 +(D^{2}\a^{-1}\b^{2}(v)\cdot[\a^{-1}\b(y),u])\cdot[\b(x),\a(z)] -(D^{2}\a^{-1}\b^{2}(v)\cdot[\a^{-1}\b(z),u])\\
 &&\hspace{-2mm}\cdot[\b(x),\a(y)]\Big)+\Big(\hspace{-1.5mm}-(D\a^{-1}\b^{2}(u)\cdot[D\a^{-1}\b(v),x])
 \cdot[\b(y),\a(z)]-(D\a^{-1}\b^{2}(v)\cdot[\a^{-1}\b(x),D(u)])\cdot[\b(y),\\
 &&\hspace{-2mm}\a(z)]+(D\a^{-1}\b^{2}(u)\cdot[D\a^{-1}\b(v),y])\cdot[\b(x),\a(z)]
 +(D\a^{-1}\b^{2}(v)\cdot[\a^{-1}\b(y),D(u)])\cdot[\b(x),\a(z)]\\
 &&\hspace{-2mm}-(D\a^{-1}\b^{2}(u)\cdot[D\a^{-1}\b(v),z])\cdot[\b(x),\a(y)]
 -(D\a^{-1}\b^{2}(v)\cdot[\a^{-1}\b(z),D(u)])\cdot[\b(x),\a(y)]\Big)+\Big(D\b^{2}(u)\\
 &&\hspace{-2mm}\cdot[\b(v),D(x)\cdot[\b^{-1}(y),\a\b^{-2}(z)]]-
 (D\a^{-1}\b^{2}(u)\cdot[\a^{-1}\b(v),D(x)])\cdot[\b(y),\a(z)]-(D\a^{-1}\b^{2}(x)\\
 &&\hspace{-2mm}\cdot[D\a^{-1}\b(u),v])\cdot[\b(y),\a(z)]
 +D\b^{2}(x)\cdot[\b(z),D(u)\cdot[\b^{-1}(v),\a\b^{-2}(y)]]
 -D\b^{2}(x)\cdot[\b(y),D(u)\cdot[\b^{-1}(v),\\
 &&\hspace{-2mm}\a\b^{-2}(z)]]\Big) +\Big(D\b^{2}(v)\cdot[D\a^{-1}\b(x)\cdot[\a^{-1}(y),\b^{-1}(z)],\a(u)]
 -(D\a^{-1}\b^{2}(v)\cdot[D\a^{-1}\b(x),u])\cdot[\b(y),\a(z)]\\
 &&\hspace{-2mm}-(D\a^{-1}\b^{2}(x)\cdot[\a^{-1}\b(u),D(v)])\cdot[\b(y),\a(z)]
 +D\b^{2}(x)\cdot[\b(z),D(v)\cdot[\b^{-1}(y),\a\b^{-2}(u)]]-D\b^{2}(x)\\
 &&\hspace{-2mm}\cdot[\b(y),D(v)\cdot[\b^{-1}(z),\a\b^{-2}(u)]]\Big) +\Big(D\b^{2}(u)\cdot[\b(v),D(y)\cdot[\b^{-1}(z),\a\b^{-2}(x)]]
 -D\b^{2}(y)\cdot[\b(z),D(u)\\
 &&\hspace{-2mm}\cdot[\b^{-1}(v),\a\b^{-2}(x)]]+(D\a^{-1}\b^{2}(u)\cdot[\a^{-1}\b(v),D(y)])\cdot[\b(x),\a(z)]
 +(D\a^{-1}\b^{2}(y)\cdot[D\a^{-1}\b(u),v])\cdot[\b(x),\\
 &&\hspace{-2mm}\a(z)]-D\b^{2}(y)\cdot[D\a^{-1}\b(u)\cdot[\a^{-1}(v),\b^{-1}(z)],\a(x)]\Big) +\Big(D\b^{2}(v)\cdot[D\a^{-1}\b(y)\cdot[\a^{-1}(z),\b^{-1}(x)],\a(u)]\\
 &&\hspace{-2mm}-D\b^{2}(y)\cdot[\b(z),D(v)\cdot[\b^{-1}(x),\a\b^{-2}(u)]]
 +(D\a^{-1}\b^{2}(v)\cdot[D\a^{-1}\b(y),u])\cdot[\b(x),\a(z)]
 +(D\a^{-1}\b^{2}(y)\\
 &&\hspace{-2mm}\cdot[\a^{-1}\b(u),D(v)])\cdot[\b(x),\a(z)]
 -D\b^{2}(y)\cdot[D\a^{-1}\b(v)\cdot[\a^{-1}(z),\b^{-1}(u)],\a(x)]\Big)
 +\Big(D\b^{2}(u)\cdot[\b(v),D(z)\\
 &&\hspace{-2mm}\cdot[\b^{-1}(x),\a\b^{-2}(y)]]
 +D\b^{2}(z)\cdot[D\a^{-1}\b(u)\cdot[\a^{-1}(v),\b^{-1}(x)],\a(y)]
 -D\b^{2}(z)\cdot[D\a^{-1}\b(u)\cdot[\a^{-1}(v),\\
 &&\hspace{-2mm}\b^{-1}(y)],\a(x)]
 -(D\a^{-1}\b^{2}(u)\cdot[\a^{-1}\b(v),D(z)])\cdot[\b(x),\a(y)] -(D\a^{-1}\b^{2}(z)\cdot[D\a^{-1}\b(u),v])\cdot[\b(x),\\
 &&\hspace{-2mm}\a(y)]\Big)+\Big(D\b^{2}(v)\cdot[D\a^{-1}\b(z)\cdot[\a^{-1}(x),\b^{-1}(y)],\a(u)]
 -D\b^{2}(z)\cdot[D\a^{-1}\b(v)\cdot[\a^{-1}(x),\b^{-1}(u)],\a(y)]\\
 &&\hspace{-2mm}+D\b^{2}(z)\cdot[D\a^{-1}\b(v)\cdot[\a^{-1}(y),\b^{-1}(u)],\a(x)]
 -(D\a^{-1}\b^{2}(v)\cdot[D\a^{-1}\b(z),u])\cdot[\b(x),\a(y)]\\
 &&\hspace{-2mm}-(D\a^{-1}\b^{2}(z)\cdot[\a^{-1}\b(u),D(v)])\cdot[\b(x),\a(y)]\Big)
 +\Big((D\a^{-1}\b^{2}(x)\cdot[D\a^{-1}\b(y),z])\cdot[\b(u),\a(v)]\\
 &&\hspace{-2mm}
 +(D\a^{-1}\b^{2}(y)\cdot[\a^{-1}\b(z),D(x)])\cdot[\b(u),\a(v)]
 -D\b^{2}(y)\cdot[\b(z),D(x)\cdot[\b^{-1}(u),\a\b^{-2}(v)]]
 +D\b^{2}(x)\\
 &&\hspace{-2mm}\cdot[\b(z),D(y)\cdot[\b^{-1}(u),\a\b^{-2}(v)]]\Big)
 +\Big((D\a^{-1}\b^{2}(x)\cdot[\a^{-1}\b(y),D(z)])\cdot[\b(u),\a(v)]
 -D\b^{2}(z)\cdot[D\a^{-1}\b(x)\\
 &&\hspace{-2mm}\cdot[\a^{-1}(u),\b^{-1}(v)],\a(y)]
 -D\b^{2}(x)\cdot[\b(y),D(z)\cdot[\b^{-1}(u),\a\b^{-2}(v)]]
 +(D\a^{-1}\b^{2}(z)\cdot[D\a^{-1}\b(x),y])\\
 &&\hspace{-2mm}\cdot[\b(u),\a(v)]\Big) +\Big((D\a^{-1}\b^{2}(z)\cdot[\a^{-1}\b(x),D(y)])\cdot[\b(u),\a(v)]
 +(D\a^{-1}\b^{2}(y)\cdot[D\a^{-1}\b(z),x])\cdot[\b(u),\\
 &&\hspace{-2mm}\a(v)]+D\b^{2}(z)\cdot[D\a^{-1}\b(y)\cdot[\a^{-1}(u),\b^{-1}(v)],\a(x)]
 +D\b^{2}(y)\cdot[D\a^{-1}\b(z)\cdot[\a^{-1}(u),\b^{-1}(v)],\a(x)]\Big)\\
 &\stackrel{\bigtriangleup}=&\hspace{-2mm}\hbox{(I)+(II)+(III)+(IV)+(V)+(VI)+(VII)+(VIII)+(IX)
 +(X)+(XI)+(XII)+(XIII)+(XIV)}.
 \end{eqnarray*}
 While
 \begin{eqnarray*}
 \hbox{(I)}\hspace{-2mm}&\stackrel{(\ref{eq:4.1})(\ref{eq:4.2})(\ref{eq:4.4})}=&\hspace{-2mm}0,\\
 \hbox{(II)}\hspace{-2mm}&\stackrel{(\ref{eq:4.1})(\ref{eq:4.2})(\ref{eq:4.4})}=&
 \hspace{-2mm}-D^{2}\b^{2}(u)\cdot([\a^{-1}\b(v),x]\cdot[y,\a\b^{-1}(z)]
 +[\a^{-1}\b(v),y]\cdot[z,\a\b^{-1}(x)] +[\a^{-1}\b(v),z]\\
 &&\hspace{-2mm}\cdot[x,\a\b^{-1}(y)])-D^{2}\b^{2}(v)\cdot([\a^{-1}\b(x),u]\cdot[y,\a\b^{-1}(z)]+[\a^{-1}\b(y),u]\cdot[z,\a\b^{-1}(x)]\\
 &&\hspace{-2mm}+[\a^{-1}\b(z),u]\cdot[x,\a\b^{-1}(y)])
 \stackrel{(\ref{eq:4.2})(\ref{eq:4.4})(\ref{eq:5.6})}=0.
 \end{eqnarray*}}
 Similarly, $\hbox{(III)}=0$.{\small
 \begin{eqnarray*}
 \hbox{(IV)}\hspace{-10mm}&=&\hspace{-10mm}D\b^{2}(u)\cdot[\b(v),D(x)
 \cdot[\b^{-1}(y),\a\b^{-2}(z)]] -(D\a^{-1}\b^{2}(u)\cdot[\a^{-1}\b(v),D(x)])\cdot[\b(y),\a(z)]\\
 &&\hspace{-10mm}-(D\a^{-1}\b^{2}(x)\cdot[D\a^{-1}\b(u),v])\cdot[\b(y),\a(z)] +D\b^{2}(x)\cdot[\b(z),D(u)\cdot[\b^{-1}(v),\a\b^{-2}(y)]]\\
 &&\hspace{-10mm}-D\b^{2}(x)\cdot[\b(y),D(u)\cdot[\b^{-1}(v),\a\b^{-2}(z)]]\\
 &\stackrel{(\ref{eq:4.2})(\ref{eq:4.6})}=&\hspace{-6mm}
 D\b^{2}(u)\cdot([\a^{-1}\b(v),D(x)]\cdot[y,\a\b^{-1}(z)]) +D\b^{2}(u)\cdot(D\b(x)\cdot[v,[\b^{-1}(y),\a\b^{-2}(z)]])\\
 &&\hspace{-8mm}-D\b^{2}(u)\cdot([\a^{-1}\b(v),D(x)]\cdot[y,\a\b^{-1}(z)]) -D\b^{2}(x)\cdot([D\a^{-1}\b(u),v]\cdot[y,\a\b^{-1}(z)])\\
 &&\hspace{-8mm}+D\b^{2}(x)\cdot([\a^{-1}\b(z),D(u)]\cdot[v,\a\b^{-1}(y)]) +D\b^{2}(x)\cdot(D\b(u)\cdot[z,[\b^{-1}(v),\a\b^{-2}(y)]])\\
 &&\hspace{-8mm}-D\b^{2}(x)\cdot([\a^{-1}\b(y),D(u)]\cdot[v,\a\b^{-1}(z)]) -D\b^{2}(x)\cdot(D\b(u)\cdot[y,[\b^{-1}(v),\a\b^{-2}(z)]])\\
 &\stackrel{(\ref{eq:4.1})(\ref{eq:4.2})(\ref{eq:4.4})}=&\hspace{-4mm}
 -D\b^{2}(x)\cdot([D\a^{-1}\b(u),v]\cdot[y,\a\b^{-1}(z)]
 +[\a^{-1}\b(z),D(u)]\cdot[y,\a\b^{-1}(v)]\\
 &&\hspace{-8mm}+[\a^{-1}\b(y),D(u)]\cdot[v,\a\b^{-1}(z)]) +(D\a^{-1}\b^{2}(u)\cdot D(x))\cdot([\b(v),[y,\a\b^{-1}(z)]]\\
 &&\hspace{-8mm}+[\b(z),[v,\a\b^{-1}(y)]]+[\b(y),[z,\a\b^{-1}(v)]])
 \stackrel{(\ref{eq:4.2})(\ref{eq:4.4})(\ref{eq:4.7})(\ref{eq:5.6})}=0.
 \end{eqnarray*}}
 Likely, $\hbox{(V)=(VI)=(VII)=(VIII)=(IX)=(X)=(XI)}=0$. {\small
 \begin{eqnarray*}
 \hbox{(XII)}\hspace{-6mm}
 &\stackrel{(\ref{eq:4.1})(\ref{eq:4.6})}=&\hspace{-6mm}
 D\b^{2}(x)\cdot([D\a^{-1}\b(y),z]\cdot[u,\a\b^{-1}(v)]) +D\b^{2}(y)\cdot([\a^{-1}\b(z),D(x)]\cdot[u,\a\b^{-1}(v)])\\
 &&-D\b^{2}(y)\cdot([\a^{-1}\b(z),D(x)]\cdot[u,\a\b^{-1}(v)]) -D\b^{2}(y)\cdot(D\b(x)\cdot[z,[\b^{-1}(u),\a\b^{-2}(v)]]\\
 &&+D\b^{2}(x)\cdot([\a^{-1}\b(z),D(y)]\cdot[u,\a\b^{-1}(v)]) +D\b^{2}(x)\cdot(D\b(y)\cdot[z,[\b^{-1}(u),\a\b^{-2}(v)]]\\
 &\stackrel{(\ref{eq:4.1})(\ref{eq:4.2})(\ref{eq:4.4})}=&0
 \end{eqnarray*}}
 By the same method, $\hbox{(XIII)=(XIV)}=0$. Thus, Eq.(\ref{eq:7.4}) holds.

 {\bf Step 2.} Eq.(\ref{eq:7.6}) can be checked as follows.{\small
 \begin{eqnarray*}
 [\a\b(x),\a\b(y),u\cdot v]
 \hspace{-9mm}&=&\hspace{-9mm}D\a\b(x)\cdot[\a(y),\b^{-1}(u\cdot v)]+D\a\b(y)\cdot[\a^{-1}(u\cdot v),\a^{2}\b^{-1}(x)]+(D\a^{-1}\b(u)\\
 &&\hspace{-4mm}\cdot\a^{-1}\b(v))\cdot[\a(x),\a^{2}\b^{-1}(y)]+
 (\a^{-1}\b(u)\cdot D\a^{-1}\b(v))\cdot[\a(x),\a^{2}\b^{-1}(y)]\\
 &\stackrel{(\ref{eq:4.4})(\ref{eq:4.6})}=&\hspace{-4mm}
 D\a\b(x)\cdot([y,\b^{-1}(u)]\cdot v+u\cdot [\a\b^{-1}(y),\b^{-1}(v)])-D\a\b(y)\cdot([x,\b^{-1}(u)]\cdot v\\
 &&\hspace{-4mm}+u\cdot [\a\b^{-1}(x),\b^{-1}(v)])+(D\a^{-1}\b(u)\cdot\a^{-1}\b(v))\cdot[\a(x),\a^{2}\b^{-1}(y)]\\
 &&\hspace{-4mm}+(\a^{-1}\b(u)\cdot D\a^{-1}\b(v))\cdot[\a(x),\a^{2}\b^{-1}(y)]\\
 &\stackrel{(\ref{eq:4.1})(\ref{eq:4.2})(\ref{eq:4.4})}=&\hspace{-2mm}(D\b(x)\cdot[y,\b^{-1}(u)])\cdot \b(v)+\b(u)\cdot (D\a(x)\cdot[\a\b^{-1}(y),\b^{-1}(v)])\\
 &&\hspace{-6mm}+(D\b(y)\cdot[\a^{-1}(u),\a\b^{-1}(x)])\cdot \b(v)+\b(u)\cdot (D\a(y)\cdot[\a^{-1}(v),\a^{2}\b^{-2}(x)])\\
 &&\hspace{-6mm}+(D\a^{-1}\b(u)\cdot[x,\a\b^{-1}(y)])\cdot \b(v)+
 \b(u)\cdot (D\a^{-1}\b(v)\cdot[\a\b^{-1}(x),\a^{2}\b^{-2}(y)])\\
 &=&\hspace{-9mm}[\b(x),\b(y),u]\cdot\b(v)+\b(u)\cdot[\a(x),\a(y),v].
 \end{eqnarray*}}
 \vskip-5mm
 {\bf Step 3.} Now we prove Eq.(\ref{eq:7.7}). {\small
 \begin{eqnarray*}
 &&\hspace{-4mm}-[\b^{2}(x),\b^{2}(y),\a\b(z)]\cdot[\a\b(u),\a\b(v),\a^{2}(w)] +[\b^{2}(x),\b^{2}(y),\a\b(u)]\cdot[\a\b(z),\a\b(v),\a^{2}(w)]\\
 &&-[\b^{2}(x),\b^{2}(y),\a\b(v)]\cdot[\a\b(z),\a\b(u),\a^{2}(w)] +[\b^{2}(x),\b^{2}(y),\a\b(w)]\cdot[\a\b(z),\a\b(u),\a^{2}(v)]\\
 &&\hspace{-4mm}=\Big(\hspace{-1.5mm}-(D\b^{2}(x)\cdot[\b(y),\a(z)])
 \cdot(D\a\b(u)\cdot[\a(v),\a^{2}\b^{-1}(w)])-(D\b^{2}(x)\cdot[\b(y),\a(v)])
 \cdot(D\a\b(u)\\
 &&\cdot[\a(w),\a^{2}\b^{-1}(z)])
 +(D\b^{2}(x)\cdot[\b(y),\a(w)])\cdot(D\a\b(u)\cdot[\a(v),\a^{2}\b^{-1}(z)])\Big)
 +\Big(\hspace{-1.5mm}-(D\b^{2}(x)\\
 &&\cdot[\b(y),\a(z)])\cdot(D\a\b(v)\cdot[\a(w),\a^{2}\b^{-1}(u)])
 +(D\b^{2}(x)\cdot[\b(y),\a(u)])\cdot(D\a\b(v)\cdot[\a(w),\\
 &&\a^{2}\b^{-1}(z)])+(D\b^{2}(x)\cdot[\b(y),\a(w)])\cdot(D\a\b(v)
 \cdot[\a(z),\a^{2}\b^{-1}(u)])\Big)+\Big(\hspace{-1.5mm}-(D\b^{2}(x)\cdot[\b(y),\\
 &&\a(z)])\cdot(D\a\b(w)\cdot[\a(u),\a^{2}\b^{-1}(v)])
 +(D\b^{2}(x)\cdot[\b(y),\a(u)])\cdot(D\a\b(w)\cdot[\a(z),\a^{2}\b^{-1}(v)])\\
 &&-(D\b^{2}(x)\cdot[\b(y),\a(v)])\cdot(D\a\b(w)\cdot[\a(z),\a^{2}\b^{-1}(u)])\Big)
 +\Big(\hspace{-1.5mm}-(D\b^{2}(y)\cdot[\b(z),\a(x)])\cdot(D\a\b(u)\\
 &&\cdot[\a(v),\a^{2}\b^{-1}(w)])-(D\b^{2}(y)\cdot[\b(v),\a(x)])\cdot(D\a\b(u)
 \cdot[\a(w),\a^{2}\b^{-1}(z)])+(D\b^{2}(y)\cdot[\b(w),\\
 &&\a(x)])\cdot(D\a\b(u)\cdot[\a(v),\a^{2}\b^{-1}(z)])\Big)
 +\Big(\hspace{-1.5mm}-(D\b^{2}(y)\cdot[\b(z),\a(x)])\cdot(D\a\b(v)
 \cdot[\a(w),\a^{2}\b^{-1}(u)])\\
 &&+(D\b^{2}(y)\cdot[\b(u),\a(x)])\cdot(D\a\b(v)\cdot[\a(w),\a^{2}\b^{-1}(z)]
 +(D\b^{2}(y)\cdot[\b(w),\a(x)])\cdot(D\a\b(v)\\
 &&\cdot[\a(z),\a^{2}\b^{-1}(u)])\Big)
 +\Big(\hspace{-1.5mm}-(D\b^{2}(y)\cdot[\b(z),\a(x)])\cdot(D\a\b(w)
 \cdot[\a(u),\a^{2}\b^{-1}(v)])+(D\b^{2}(y)\cdot[\b(u),\\
 &&\a(x)])\cdot(D\a\b(w)\cdot[\a(z),\a^{2}\b^{-1}(v)])
 -(D\b^{2}(y)\cdot[\b(v),\a(x)])\cdot(D\a\b(w)\cdot[\a(z),\a^{2}\b^{-1}(u)])\Big)\\
 &&+\Big((D\b^{2}(x)\cdot[\b(y),\a(u)])\cdot(D\a\b(z)\cdot[\a(v),\a^{2}\b^{-1}(w)])
 -(D\b^{2}(x)\cdot[\b(y),\a(v)])\cdot(D\a\b(z)\\
 &&\cdot[\a(u),\a^{2}\b^{-1}(w)])
 +(D\b^{2}(x)\cdot[\b(y),\a(w)])\cdot(D\a\b(z)\cdot[\a(u),\a^{2}\b^{-1}(v)])\Big)+\Big((D\b^{2}(y)\cdot[\b(u),\\
 &&\a(x)])\cdot(D\a\b(z)\cdot[\a(v),\a^{2}\b^{-1}(w)])
 -(D\b^{2}(y)\cdot[\b(v),\a(x)])\cdot(D\a\b(z)\cdot[\a(u),\a^{2}\b^{-1}(w)])\\
 &&+(D\b^{2}(y)\cdot[\b(w),\a(x)])\cdot(D\a\b(z)\cdot[\a(u),\a^{2}\b^{-1}(v)])\Big)+\Big(\hspace{-1.5mm}-(D\b^{2}(z)\cdot[\b(x),\a(y)])\cdot(D\a\b(u)\\
 &&\cdot[\a(v),\a^{2}\b^{-1}(w)])
 +(D\b^{2}(u)\cdot[\b(x),\a(y)])\cdot(D\a\b(z)\cdot[\a(v),\a^{2}\b^{-1}(w)])\Big)+\Big(\hspace{-1.5mm}-(D\b^{2}(z)\cdot[\b(x),\\
 &&\a(y)])\cdot(D\a\b(v)\cdot[\a(w),\a^{2}\b^{-1}(u)])
 -(D\b^{2}(v)\cdot[\b(x),\a(y)])\cdot(D\a\b(z)\cdot[\a(u),\a^{2}\b^{-1}(w)])\Big)\\
 &&+\Big(\hspace{-1.5mm}-(D\b^{2}(z)\cdot[\b(x),\a(y)])\cdot(D\a\b(w)
 \cdot[\a(u),\a^{2}\b^{-1}(v)])+(D\b^{2}(w)\cdot[\b(x),\a(y)])\cdot(D\a\b(z)\\
 &&\cdot[\a(u),\a^{2}\b^{-1}(v)])\Big)
 +\Big((D\b^{2}(u)\cdot[\b(x),\a(y)])\cdot(D\a\b(v)\cdot[\a(w),\a^{2}\b^{-1}(z)])-(D\b^{2}(v)\cdot[\b(x),\\
 &&\a(y)])\cdot(D\a\b(u)\cdot[\a(w),\a^{2}\b^{-1}(z)])\Big)
 +\Big((D\b^{2}(u)\cdot[\b(x),\a(y)])\cdot(D\a\b(w)\cdot[\a(z),\a^{2}\b^{-1}(v)])\\
 &&+(D\b^{2}(w)\cdot[\b(x),\a(y)])\cdot(D\a\b(u)\cdot[\a(v),\a^{2}\b^{-1}(z)])\Big)+\Big(\hspace{-1.5mm}-(D\b^{2}(v)\cdot[\b(x),\a(y)])\cdot(D\a\b(w)\\
 &&\cdot[\a(z),\a^{2}\b^{-1}(u)])
 +(D\b^{2}(w)\cdot[\b(x),\a(y)])\cdot(D\a\b(v)\cdot[\a(z),\a^{2}\b^{-1}(u)])\Big)\\
 &&\hspace{-4mm}\stackrel{\bigtriangleup}=\hbox{(I)}+\hbox{(II)}+\hbox{(III)}
 +\hbox{(IV)}+\hbox{(V)}+\hbox{(VI)}+\hbox{(VII)}+ \hbox{(VIII)}+\hbox{(IX)}+\hbox{(X)}+\hbox{(XI)}+\hbox{(XII)}+\hbox{(XIII)}+\hbox{(XIV)}.
 \end{eqnarray*} }
 Further,{\small
 \begin{eqnarray*}
 \hbox{(I)}
 &=&-(D\b^{2}(x)\cdot[\b(y),\a(z)])\cdot(D\a\b(u)\cdot[\a(v),\a^{2}\b^{-1}(w)])\\
 &&-(D\b^{2}(x)\cdot[\b(y),\a(v)])\cdot(D\a\b(u)\cdot[\a(w),\a^{2}\b^{-1}(z)])\\
 &&+(D\b^{2}(x)\cdot[\b(y),\a(w)])\cdot(D\a\b(u)\cdot[\a(v),\a^{2}\b^{-1}(z)]) \stackrel{(\ref{eq:4.1})(\ref{eq:4.2})(\ref{eq:4.4})(\ref{eq:5.6})}= 0.
 \end{eqnarray*}}
 Similarly, \hbox{(II)}=\hbox{(III)}=\hbox{(IV)}=\hbox{(V)}=\hbox{(VI)}=\hbox{(VII)}=\hbox{(VIII)}=0. On the other hand, {\small
 \begin{eqnarray*}
 \hbox{(IX)}&=&-(D\b^{2}(z)\cdot[\b(x),\a(y)])\cdot(D\a\b(u)\cdot[\a(v),\a^{2}\b^{-1}(w)])\\
 &&+(D\b^{2}(u)\cdot[\b(x),\a(y)])\cdot(D\a\b(z)\cdot[\a(v),\a^{2}\b^{-1}(w)]) \stackrel{(\ref{eq:4.1})(\ref{eq:4.2})}= 0.
 \end{eqnarray*}}
 Likewise, \hbox{(X)}=\hbox{(XI)}=\hbox{(XII)}=\hbox{(XIII)}=\hbox{(XIV)}=0. The proof is finished by $(L, \cdot, \a, \b)$ is a BiHom-commutative algebra.
 \end{proof}

 \begin{pro}\label{pro:7.14} Let $(L, \cdot, [,], \a, \b)$ be a regular strong BP algebra. Define a ternary operation on $L$ by{\small
 \begin{eqnarray}
 &\hspace{-10mm}[x,y,z]=x\cdot[\b^{-1}(y),\b^{-1}(z)]+y\cdot[\a^{-1}(z),\a\b^{-2}(x)]
 +\a^{-1}\b(z)\cdot[\b^{-1}(x),\a\b^{-2}(y)], x, y, z\in L.&\label{eq:7.14}
 \end{eqnarray}}
 Then $(L, [\ ,\ ,\ ], \a, \b)$ is a 3-BiHom-Lie algebra satisfying Eq.(\ref{eq:7.7}).
 \end{pro}

 \begin{proof}
 Similar to Theorem \ref{thm:7.7}.
 \end{proof}

 \begin{rmk}\label{rmk:7.16} (1) Under the assumption of Proposition \ref{pro:7.14}, $(L, [\ ,\ ,\ ], \a, \b)$ is not a BP 3-Lie algebra since Eq.(\ref{eq:7.6}) does not hold.

 (2) Let $(L, \cdot, [,], \a, \b)$ be a BP algebra. If the strongness condition Eq.(\ref{eq:5.6}) is replaced by Eq.(\ref{eq:5.4}), then Proposition \ref{pro:7.14} still holds. In fact, by Eq.(\ref{eq:5.4}),{\small
 %Eq.(\ref{eq:7.14}) still defines a 3-BiHom-Lie algebra by a straightforward proof.
 \begin{eqnarray*}
 &&\hspace{-25mm}[\a\b^{2}(h)\cdot[\a\b^{2}(x),\a^{2}\b(y)],\a^{3}\b^{2}(z)] +[\a\b^{2}(h)\cdot[\a\b^{2}(y),\a^{2}\b(z)],\a^{3}\b^{2}(x)]\\
 &&+[\a\b^{2}(h)\cdot[\a\b^{2}(z),\a^{2}\b(x)],\a^{3}\b^{2}(y)]\\
 &\stackrel{(\ref{eq:4.4})(\ref{eq:4.6})}=&\hspace{-6mm}
 -[\a\b^{3}(z),\a^{2}\b(h)]\cdot[\a^{2}\b^{2}(x),\a^{3}\b(y)] -\a^{2}\b^{2}(h)\cdot[\a^{2}\b^{2}(z),[\a^{2}\b(x),\a^{3}(y)]]\\
 &&\hspace{-6mm}-[\a\b^{3}(x),\a^{2}\b(h)]\cdot[\a^{2}\b^{2}(y),\a^{3}\b(z)] -\a^{2}\b^{2}(h)\cdot[\a^{2}\b^{2}(x),[\a^{2}\b(y),\a^{3}(z)]]\\
 &&\hspace{-6mm}-[\a\b^{3}(y),\a^{2}\b(h)]\cdot[\a^{2}\b^{2}(z),\a^{3}\b(x)] -\a^{2}\b^{2}(h)\cdot[\a^{2}\b^{2}(y),[\a^{2}\b(z),\a^{3}(x)]]\\
 &\stackrel{(\ref{eq:4.2})(\ref{eq:4.4})(\ref{eq:4.7})}=&\hspace{-3mm}
 [\a\b^{3}(x),\a^{2}\b^{2}(y)]\cdot[\a^{2}\b(h),\a^{3}\b(z)] +[\a\b^{3}(y),\a^{2}\b^{2}(z)]\cdot[\a^{2}\b(h),\a^{3}\b(x)]\\
 &&+[\a\b^{3}(z),\a^{2}\b^{2}(x)]\cdot[\a^{2}\b(h),\a^{3}\b(y)],
 \end{eqnarray*}}
 as desired.
 \end{rmk}

\subsection{TBP 3-Lie algebras from TBP algebras and strong BP algebras}

 \begin{defi}\label{de:7.9} A {\bf TBP 3-Lie algebra} is a 5-tuple $(L, \cdot, [\ ,\ ,\ ], \a, \b)$, where $(L, \cdot, \a, \b)$ is a BiHom-commutative algebra and $(L, [\ ,\ ,\ ], \a, \b)$ is a
 3-BiHom-Lie algebra satisfying the following condition:
 \begin{eqnarray}
 &3\a\b(u)\cdot[x,y,z]=[\b(u)\cdot x,\b(y),\b(z)]+[\b(x),\b(u)\cdot y,\b(z)]
 +[\b(x),\b(y),\a(u)\cdot z].&\label{eq:7.10}
 \end{eqnarray}
 \end{defi}

 \begin{pro}\label{pro:7.10} Let $(L, \cdot, [\ ,\ ,\ ])$ be a transposed Possion 3-Lie algebra and $\a,\b:L\lr L$ two commuting algebra maps such that $\a(x\cdot y)=\a(x)\cdot \a(y), \b(x\cdot y)=\b(x)\cdot \b(y), \a[x, y, z]=[\a(x), \a(y), \a(z)], \b[x, y, z]=[\b(x), \b(y), \b(z)]$. Then $(L, \cdot'=\cdot\circ(\a\o\b), [\ ,\ ,\ ]'=[\ ,\ ,\ ]\circ(\a\o\a\o\b), \a, \b)$ is a TBP 3-Lie algebra, called the ``Yau twist" of $(L, \cdot, [\ ,\ ,\ ])$.
 \end{pro}

 \begin{proof} We only verify that Eq.(\ref{eq:7.10}) holds for $(L, \cdot', [\ ,\ ,\ ]', \a, \b)$ and others can be obtained by \cite[Claim 3.7]{GMMP} and \cite[Theorem 1.12]{KMS}.
 \begin{eqnarray*}
 &&\hspace{-10mm}[\b(u)\cdot' x, \b(y), \b(z)]'+[\b(x), \b(u)\cdot' y, \b(z)]'
 +[\b(x), \b(y), \a(u)\cdot' z]'\\
 &=&[\a^{2}\b(u)\cdot \a\b(x), \a\b(y), \b^{2}(z)]+[\a\b(x),\a^{2}\b(u)\cdot \a\b(y), \b^{2}(z)] +[\a\b(x), \a\b(y), \a^{2}\b(u)\cdot \b^{2}(z)]\\
 &=&3\a^{2}\b(u)\cdot[\a\b(x), \a\b(y), \b^{2}(z)]\\
 &=&3\a\b(u)\cdot'[x,y,z]',
 \end{eqnarray*}
 as we needed.
 \end{proof}

 \begin{pro}\label{pro:7.17} Let $(L, \cdot, \a, \b)$ be a BiHom-commutative algebra and $(L, [\ ,\ ,\ ], \a, \b)$ be a 3-BiHom-Lie algebra. Assume that $(L, \cdot, [\ ,\ ,\ ], \a, \b)$ is both a BP 3-Lie algebra and a TBP 3-Lie algebra. Then
 \begin{eqnarray}
 &\a\b^{2}(u)\cdot[\a\b(x),\a\b(y),\a^{2}(z)]
 =[\a\b^{2}(x),\a\b^{2}(y),\a\b(u)\cdot\a^{2}(z)]=0.&\label{eq:7.15}
 \end{eqnarray}
 \end{pro}

 \begin{proof} Let $u, x, y, z\in L$. By Eq.(\ref{eq:7.10}), we have {\small
 \begin{eqnarray*}
 3\a\b^{2}(u)\cdot[\a\b(x),\a\b(y),\a^{2}(z)]
 &=&[\b^{2}(u)\cdot\a\b(x),\a\b^{2}(y),\a^{2}\b(z)]
 +[\a\b^{2}(x),\b^{2}(u)\cdot\a\b(y),\a^{2}\b(z)]\\
 &&+[\a\b^{2}(x),\a\b^{2}(y),\a\b(u)\cdot\a^{2}\b(z)]\\
 &\stackrel{(\ref{eq:7.3})(\ref{eq:7.6})}=&-[\b^{2}(z),\b^{2}(y),\a\b(u)]\cdot\a^{2}\b(x)-
 \a\b^{2}(u)\cdot[\a\b(z),\a\b(y),\a^{2}(x)]\\
 &&-[\b^{2}(x),\b^{2}(z),\a\b(u)]\cdot\a^{2}\b(y)-
 \a\b^{2}(u)\cdot[\a\b(x),\a\b(z),\a^{2}(y)]\\
 &&+[\b^{2}(x),\b^{2}(y),\a\b(u)]\cdot\a^{2}\b(z)+
 \a\b^{2}(u)\cdot[\a\b(x),\a\b(y),\a^{2}(z)]\\
 &\stackrel{(\ref{eq:4.2})(\ref{eq:7.3})}=&3\a\b^{2}(u)\cdot[\a\b(x),\a\b(y),\a^{2}(z)]
 -[\b^{2}(z),\b^{2}(y),\a\b(u)]\cdot\a^{2}\b(x)\\
 &&-[\b^{2}(x),\b^{2}(z),\a\b(u)]\cdot\a^{2}\b(y)
 +[\b^{2}(x),\b^{2}(y),\a\b(u)]\cdot\a^{2}\b(z).
 \end{eqnarray*}}
 So {\small
 \begin{eqnarray}
 -[\b^{2}(z),\b^{2}(y),\a\b(u)]\cdot\a^{2}\b(x)
 -[\b^{2}(x),\b^{2}(z),\a\b(u)]\cdot\a^{2}\b(y) +[\b^{2}(x),\b^{2}(y),\a\b(u)]\cdot\a^{2}\b(z)=0.\label{eq:a-1-1}
 \end{eqnarray}}
 Then {\small
 \begin{eqnarray*}
 0\hspace{-4mm}&\stackrel{(\ref{eq:a-1-1})}=&\hspace{-4mm}
 -3[\b^{2}(z),\b^{2}(y),\a\b(u)]\cdot\a^{2}\b(x)
 -3[\b^{2}(x),\b^{2}(z),\a\b(u)]\cdot\a^{2}\b(y) +3[\b^{2}(x),\b^{2}(y),\a\b(u)]\cdot\a^{2}\b(z)\\
 &\stackrel{(\ref{eq:4.2})}=&\hspace{-4mm}-3\a\b^{2}(x)\cdot[\a\b(z),\a\b(y),\a^{2}(u)]
 -3\a\b^{2}(y)\cdot[\a\b(x),\a\b(z),\a^{2}(u)] +3\a\b^{2}(z)\cdot[\a\b(x),\a\b(y),\a^{2}(u)]\\
 &\stackrel{(\ref{eq:7.10})}=&\hspace{-4mm}-[\b^{2}(x)\cdot\a\b(z),\a\b^{2}(y),\a^{2}\b(u)]
 -[\a\b^{2}(z),\b^{2}(x)\cdot\a\b(y),\a^{2}\b(u)] -[\a\b^{2}(z),\a\b^{2}(y),\a\b(x)\cdot\a^{2}\b(u)]\\
 &&\hspace{-4mm}-[\b^{2}(y)\cdot\a\b(x),\a\b^{2}(z),\a^{2}\b(u)]
 -[\a\b^{2}(x),\b^{2}(y)\cdot\a\b(z),\a^{2}\b(u)] -[\a\b^{2}(x),\a\b^{2}(z),\a\b(y)\cdot\a^{2}\b(u)]\\
 &&\hspace{-4mm}+[\b^{2}(z)\cdot\a\b(x),\a\b^{2}(y),\a^{2}\b(u)]
 +[\a\b^{2}(x),\b^{2}(z)\cdot\a\b(y),\a^{2}\b(u)] +[\a\b^{2}(x),\a\b^{2}(y),\a\b(z)\cdot\a^{2}\b(u)]\\
 &\hspace{-4mm}\stackrel{(\ref{eq:4.2})(\ref{eq:7.3})}=&\hspace{-8mm}
 -[\a\b^{2}(z),\a\b^{2}(y),\a\b(x)\cdot\a^{2}\b(u)]
 -[\a\b^{2}(x),\a\b^{2}(z),\a\b(y)\cdot\a^{2}\b(u)] +[\a\b^{2}(x),\a\b^{2}(y),\a\b(z)\cdot\a^{2}\b(u)]\\
 &\stackrel{(\ref{eq:4.2})(\ref{eq:7.3})}=&\hspace{-3mm}
 [\b^{2}(u)\cdot\a\b(x),\a\b^{2}(y),\a^{2}\b(z)]
 +[\a\b^{2}(x),\b^{2}(u)\cdot\a\b(y),\a^{2}\b(z)] +[\a\b^{2}(x),\a\b^{2}(y),\a\b(u)\cdot\a^{2}\b(z)]\\
 &\stackrel{(\ref{eq:7.10})}=&\hspace{-4mm}3\a\b^{2}(u)\cdot[\a\b(x),\a\b(y),\a^{2}(z)].
 \end{eqnarray*}}
 Therefore we have $\a\b^{2}(u)\cdot[\a\b(x),\a\b(y),\a^{2}(z)]=0$, and then $[\b^{2}(x),\b^{2}(y),\a\b(u)]\cdot\a^{2}\b(z)=0$.

 Based on the above two equations,
 \begin{eqnarray*}
 0&=&\a\b^{2}(u)\cdot[\a\b(x),\a\b(y),\a^{2}(z)]+
 [\b^{2}(x),\b^{2}(y),\a\b(u)]\cdot\a^{2}\b(z)\\
 &\stackrel{(\ref{eq:7.6})}=&[\a\b^{2}(x),\a\b^{2}(y),\a\b(u)\cdot\a^{2}(z)].
 \end{eqnarray*}

 Hence the proof is finished.
 \end{proof}

 \begin{rmk}\label{rmk:7.17a} Assume that the structure maps $\a, \b$ are bijective. Then Eq.(\ref{eq:7.15}) is equivalent to
 \begin{eqnarray}
 &u\cdot[x,y,z]=[u\cdot x,y,z]=0.&\label{eq:7.15a}
 \end{eqnarray}
 In this case for Proposition \ref{pro:7.17}, the condition Eq.(\ref{eq:7.15a}) is also sufficient.
 \end{rmk}

 \begin{rmk}\label{rmk:7.13} Under the assumption of Theorem \ref{thm:7.12}, $(L, \cdot, [\ ,\ ,\ ], \a, \b)$ is a TBP 3-Lie algebra if and only if $f$ satisfies
 \begin{eqnarray}
 &&\hspace{-10mm}(f\a\b(u)+\a\b(u))\cdot\Big(f(x)\cdot[\b^{-1}(y),\b^{-1}(z)]
 +f(y)\cdot[\a^{-1}(z),\a\b^{-2}(x)] \nonumber\\ &&\hspace{70mm}+f\a^{-1}\b(z)\cdot[\b^{-1}(x),\a\b^{-2}(y)]\Big)=0.\label{eq:7.13}
 \end{eqnarray}
 In fact, for all $x, y, z, u\in L$, we have{\small
 \begin{eqnarray*}
 &&\hspace{-10mm}[\b(u)\cdot x,\b(y),\b(z)]+[\b(x),\b(u)\cdot y,\b(z)]+[\b(x),\b(y),\a(u)\cdot z]\\
 &&\hspace{-6mm}=f(\b(u)\cdot x)\cdot[y,z]+f\b(y)\cdot[\a^{-1}\b(z),\a\b^{-1}(u)\cdot\a\b^{-2}(x)] +f\a^{-1}\b^{2}(z)\cdot[u\cdot\b^{-1}(x),\a\b^{-1}(y)]\\
 &&\hspace{-3mm}+f\b(x)\cdot[u\cdot\b^{-1}(y),z]+f(\b(u)\cdot y)\cdot[\a^{-1}\b(z),\a\b^{-1}(x)] +f\a^{-1}\b^{2}(z)\cdot[x,\a\b^{-1}(u)\cdot\a\b^{-2}(y)]\\
 &&\hspace{-3mm}+f\b(x)\cdot[y,\a\b^{-1}(u)\cdot\b^{-1}(z)]
 +f\b(y)\cdot[u\cdot\a^{-1}(z),\a\b^{-1}(x)] +f(\b(u)\cdot\a^{-1}\b(z))\cdot[x,\a\b^{-1}(y)]\\
 &&\hspace{-13mm}\stackrel{(\ref{eq:4.1})(\ref{eq:4.2})(\ref{eq:5.1})}=
 -f\a\b(u)\cdot\Big(f(x)\cdot[\b^{-1}(y),\b^{-1}(z)]
 +f(y)\cdot[\a^{-1}(z),\a\b^{-2}(x)]
 +f\a^{-1}\b(z)\cdot[\b^{-1}(x),\a\b^{-2}(y)]\Big)\\
 &&+2\a\b(u)\cdot\Big(f(x)\cdot[\b^{-1}(y),\b^{-1}(z)]+f(y)\cdot[\a^{-1}(z),\a\b^{-2}(x)] +f\a^{-1}\b(z)\cdot[\b^{-1}(x),\a\b^{-2}(y)]\Big)
 \end{eqnarray*}
 and
 \begin{eqnarray*}
 &&\hspace{-10mm}3\a\b(u)\cdot[x,y,z]\\
 &&=3\a\b(u)\cdot(f(x)\cdot[\b^{-1}(y),\b^{-1}(z)]
 +f(y)\cdot[\a^{-1}(z),\a\b^{-2}(x)] +f\a^{-1}\b(z)\cdot[\b^{-1}(x),\a\b^{-2}(y)]).
 \end{eqnarray*}}
 The rest is obvious.
 \end{rmk}

 \begin{pro}\label{pro:7.11} Let $(L, \cdot, [,], \a, \b)$ be a regular TBP algebra and $D$ a derivation of $(L, \cdot, \a, \b)$ and $(L, [,], \a, \b)$. Then the 5-tuple $(L, \cdot, [\ ,\ ,\ ], \a, \b)$ is a TBP 3-Lie algebra, where the 3-BiHom-Lie algebra is defined by Eq.(\ref{eq:7.8}).
 \end{pro}

 \begin{proof} To finish the proof, by Theorem \ref{thm:7.8}, we only need to verify that Eq.(\ref{eq:7.10}) holds. In fact, for all $x, y, z, u\in L$, we have{\small
 \begin{eqnarray*}
 &&\hspace{-10mm}[\b(u)\cdot x,\b(y),\b(z)]+[\b(x),\b(u)\cdot y,\b(z)]+[\b(x),\b(y),\a(u)\cdot z]\\
 &&\hspace{-3mm}=D(\b(u)\cdot x)\cdot[y,z]+D\b(y)\cdot[\a^{-1}\b(z),\a\b^{-1}(u)\cdot\a\b^{-2}(x)]
 +D\a^{-1}\b^{2}(z)\cdot[u\cdot\b^{-1}(x),\a\b^{-1}(y)]\\
 &&+D\b(x)\cdot[u\cdot\b^{-1}(y),z]+D(\b(u)\cdot y)\cdot[\a^{-1}\b(z),\a\b^{-1}(x)]
 +D\a^{-1}\b^{2}(z)\cdot[x,\a\b^{-1}(u)\cdot\a\b^{-2}(y)]\\
 &&+D\b(x)\cdot[y,\a\b^{-1}(u)\cdot\b^{-1}(z)]+D\b(y)\cdot[u\cdot\a^{-1}(z),\a\b^{-1}(x)]
 +D(\b(u)\cdot\a^{-1}\b(z))\cdot[x,\a\b^{-1}(y)]\\
 &&\hspace{-10mm}\stackrel{(\ref{eq:4.1})(\ref{eq:4.2})(\ref{eq:5.1})}=
 D\a\b(u)\cdot(x\cdot[\b^{-1}(y),\b^{-1}(z)]
 +y\cdot[\a^{-1}(z),\a\b^{-2}(x)] +\a^{-1}\b(z)\cdot[\b^{-1}(x),\a\b^{-2}(y)])\\
 &&+\a\b(u)\cdot(D(x)\cdot[\b^{-1}(y),\b^{-1}(z)]+D(y)\cdot[\a^{-1}(z),\a\b^{-2}(x)] +D\a^{-1}\b(z)\cdot[\b^{-1}(x),\a\b^{-2}(y)])\\
 &&+2\a\b(u)\cdot(D(x)\cdot[\b^{-1}(y),\b^{-1}(z)]+D(y)\cdot[\a^{-1}(z),\a\b^{-2}(x)] +D\a^{-1}\b(z)\cdot[\b^{-1}(x),\a\b^{-2}(y)])\\
 &&\hspace{-6mm}\stackrel{(\ref{eq:5.3})}=3\a\b(u)\cdot(D(x)\cdot[\b^{-1}(y),\b^{-1}(z)]
 +D(y)\cdot[\a^{-1}(z),\a\b^{-2}(x)] +D\a^{-1}\b(z)\cdot[\b^{-1}(x),\a\b^{-2}(y)])\\
 &&\hspace{-4mm}=3\a\b(u)\cdot[x,y,z],
 \end{eqnarray*}}
 finishing the proof.
 \end{proof}

 \begin{thm}\label{thm:7.15} Let $(L, \cdot, [,], \a, \b)$ be a regular BP algebra. If Eq.(\ref{eq:7.14}) defines a 3-BiHom-Lie algebra $(L, [\ ,\ ,\ ], \a, \b)$, then $(L, \cdot, [\ ,\ ,\ ], \a, \b)$ is a TBP 3-Lie algebra. In particular, if $(L,\cdot,[,],\a,\b)$ is a strong BP algebra, then the 3-BiHom-Lie algebra $(L, [\ ,\ ,\ ], \a, \b)$ given in Proposition \ref{pro:7.14} gives a TBP 3-Lie algebra $(L, \cdot, [\ ,\ ,\ ], \a, \b)$.
 \end{thm}

 \begin{proof} For all $u, x, y, z\in L$, we have{\small
 \begin{eqnarray*}
 &&\hspace{-10mm}[\b(u)\cdot x,\b(y),\b(z)]+[\b(x),\b(u)\cdot y,\b(z)]+[\b(x),\b(y),\a(u)\cdot z]\\
 &&\hspace{-3mm}=(\b(u)\cdot x)\cdot[y,z]+\b(y)\cdot[\a^{-1}\b(z),\a\b^{-1}(u)\cdot\a\b^{-2}(x)] +\a^{-1}\b^{2}(z)\cdot[u\cdot\b^{-1}(x),\a\b^{-1}(y)]\\
 &&+\b(x)\cdot[u\cdot\b^{-1}(y),z]+(\b(u)\cdot y)\cdot[\a^{-1}\b(z),\a\b^{-1}(x)] +\a^{-1}\b^{2}(z)\cdot[x,\a\b^{-1}(u)\cdot\a\b^{-2}(y)]\\
 &&+\b(x)\cdot[y,\a\b^{-1}(u)\cdot\b^{-1}(z)]+\b(y)\cdot[u\cdot\a^{-1}(z),\a\b^{-1}(x)] +(\b(u)\cdot\a^{-1}\b(z))\cdot[x,\a\b^{-1}(y)]\\
 &&\hspace{-12mm}\stackrel{(\ref{eq:4.1})(\ref{eq:4.2})(\ref{eq:4.4})(\ref{eq:4.6})}=3\a\b(u)\cdot(x\cdot[\b^{-1}(y),\b^{-1}(z)]+y\cdot[\a^{-1}(z),\a\b^{-2}(x)] +\a^{-1}\b(z)\cdot[\b^{-1}(x),\a\b^{-2}(y)])\\
 &&\hspace{-4mm}=3\a\b(u)\cdot[x,y,z].
 \end{eqnarray*}}
 Then Eq.(\ref{eq:7.10}) holds. And the proof is completed.
 \end{proof}

 For possible future research, we introduce the notion of TBP $n$-Lie algebra.
 \begin{defi}\label{de:7.18} Let $n\geq 2$ be an integer. A {\bf TBP $n$-Lie algebra} is a 5-tuple $(L, \cdot, \mu, \a, \b)$, where $(L, \cdot, \a, \b)$ is a BiHom-commutative algebra and $(L, \mu, \a, \b)$ is an $n$-BiHom-Lie algebra such that {\small
 \begin{eqnarray*}
 n\a\b(u)\cdot\mu(x_{1},x_{2},\cdot\cdot\cdot,x_{n})\hspace{-2mm}
 &=&\hspace{-2mm}\mu(\b(u)\cdot x_{1},\b(x_{2}),\cdot\cdot\cdot,\b(x_{n-1}),\b(x_{n}))+\mu(\b(x_{1}),\b(u)\cdot x_{2},\cdot\cdot\cdot,\b(x_{n-1}),\b(x_{n}))\\
 &&\hspace{-4mm}+\cdot\cdot\cdot\\
 &&\hspace{-4mm}+\mu(\b(x_{1}),\b(x_{2}),\cdot\cdot\cdot,\b(u)\cdot x_{n-1},\b(x_{n}))+\mu(\b(x_{1}),\b(x_{2}),\cdot\cdot\cdot,\b(x_{n-1}),\a(u)\cdot x_{n}),
 \end{eqnarray*}
 $\forall u, x_{1}, x_{2}, \cdot\cdot\cdot, x_{n} \in L$.}
 \end{defi}

 \begin{pro}\label{pro:7.19} Let $(L, \cdot, \mu)$ be a transposed Poisson $n$-Lie algebra \cite{BBGW}, $\a, \b: L\lr L$ be two commuting linear maps such that $\a(x\cdot y)=\a(x)\cdot \a(y), \b(x\cdot y)=\b(x)\cdot \b(y), \a\mu(x_{1},\cdot\cdot\cdot x_{n})=\mu(\a(x), \cdot\cdot\cdot, \a(x_{n})), \b\mu(x_{1},\cdot\cdot\cdot x_{n})=\mu(\b(x_{1}), \cdot\cdot\cdot, \b(x_{n}))$. Then $(L, \cdot'=\cdot\circ (\a\o\b), \mu'=\mu\circ (\a\o\a\cdot\cdot\cdot\o\a\o\b), \a, \b)$ is a TBP n-Lie algebra, called the ``Yau twist" of $(L, \cdot, \mu).$
 \end{pro}

 \begin{proof} By \cite[Claim 3.7]{GMMP} and \cite[Theorem 5.2]{KMS}, we only need to check the equality below.  For $\forall u, x_{1}, x_{2}, \cdot\cdot\cdot, x_{n}\in L$,{\small
 \begin{eqnarray*}
 &&\hspace{-8mm}\mu'(\b(u)\cdot' x_{1},\b(x_{2}),\cdot\cdot\cdot,\b(x_{n-1}),\b(x_{n}))+\mu'(\b(x_{1}),\b(u)\cdot' x_{2},\cdot\cdot\cdot,\b(x_{n-1}),\b(x_{n}))\\
 &&+\cdot\cdot\cdot\\
 &&+\mu'(\b(x_{1}),\b(x_{2}),\cdot\cdot\cdot,\b(u)\cdot'x_{n-1},\b(x_{n}))
 +\mu'(\b(x_{1}),\b(x_{2}),\cdot\cdot\cdot,\b(x_{n-1}),\a(u)\cdot' x_{n})\\
 &=&\mu(\a^{2}\b(u)\cdot \a\b(x_{1}),\a\b(x_{2}),\cdot\cdot\cdot,\a\b(x_{n-1}),\b^{2}(x_{n})) +\mu(\a\b(x_{1}),\a^{2}\b(u)\cdot \a\b(x_{2}),\cdot\cdot\cdot,\a\b(x_{n-1}),\b^{2}(x_{n}))\\
 &&+\cdot\cdot\cdot\\
 &&+\mu(\a\b(x_{1}),\a\b(x_{2}),\cdot\cdot\cdot,\a^{2}\b(u)\cdot \a\b(x_{n-1}),\b^{2}(x_{n})) +\mu(\a\b(x_{1}),\a\b(x_{2}),\cdot\cdot\cdot,\a\b(x_{n-1}),\a^{2}\b(u)\cdot \b^{2}(x_{n}))\\
 &=&n\a^{2}\b(u)\cdot\mu(\a\b(x_{1}),\a\b(x_{2}),\cdot\cdot\cdot,\a\b(x_{n-1}),\b^{2}(x_{n}))\\
 &=&n\a\b(u)\cdot'\mu'(x_{1},x_{2},\cdot\cdot\cdot,x_{n}),
 \end{eqnarray*}}
 as desired.
 \end{proof}

\section{2-dimensional TBP algebras}\label{se:ex} In this section, we list some concrete examples for TBP algebras.

 \begin{ex}\label{ex:5.4} Let $L$ be a 2-dimensional vector space with respect to a basis $\{e_{1}, e_{2}\}$. Then $(L, \cdot, [,], \a, \b)$ is a TBP algebra, where
 \begin{itemize}
 \item {\bf Case I: Lie bracket is trivial.}
 \begin{eqnarray*}
 &(1)& e_{1}\cdot e_{2}=k_{2}e_{2}, \a(e_{1})=e_{2}, \a(e_{2})=0, \b(e_{1})=k_{1}e_{1}, \b(e_{2})=k_{2}e_{2}.\\
 &(2)& e_{1}\cdot e_{1}=k_{1}e_{1}, \a(e_{1})=e_{1}, \a(e_{2})=0, \b(e_{1})=k_{1}e_{1}, \b(e_{2})=k_{2}e_{2};\\
 &(3)& e_{1}\cdot e_{1}=k_{1}e_{1}+k_{2}e_{2},  e_{1}\cdot e_{2}=(k_{1}-k_{2})e_{2},\\
 &&\a(e_{1})=e_{1}+e_{2},\a(e_{2})=0, \b(e_{1})=k_{1}e_{1}+k_{2}e_{2}, \b(e_{2})=(k_{1}-k_{2})e_{2};\\
 &(4)& e_{2}\cdot e_{2}=k_{2}e_{2}, \a(e_{1})=0, \a(e_{2})=e_{2}, \b(e_{1})=k_{1}e_{1}, \b(e_{2})=k_{2}e_{2};\\
 &(5)& e_{1}\cdot e_{1}=k_{2}e_{2}, e_{1}\cdot e_{2}=(k_{1}+k_{2})e_{2}, e_{2}\cdot e_{1}=k_{2}e_{2}, e_{2}\cdot e_{2}=(k_{1}+k_{2})e_{2},\\
 && \a(e_{1})=e_{2}, \a(e_{2})=e_{2}, \b(e_{1})=k_{1}e_{1}+k_{2}e_{2}, \b(e_{2})=(k_{1}+k_{2})e_{2};\\
 &(6)&e_{1}\cdot e_{1}=k_{1}e_{1}, e_{1}\cdot e_{2}=k_{3}e_{1}, e_{2}\cdot e_{1}=k_{2}e_{2}, e_{2}\cdot e_{2}=k_{4}e_{2},\\
 &&\a(e_{1})=e_{1}, \a(e_{2})=e_{2}, \b(e_{1})=k_{1}e_{1}+k_{2}e_{2}, \b(e_{2})=k_{3}e_{1}+k_{4}e_{2}; \\
 &(7)&e_{1}\cdot e_{1}=k_{1}e_{1}+k_{2}e_{2}, e_{1}\cdot e_{2}=k_{1}e_{2}, e_{2}\cdot e_{1}=k_{1}e_{2}, e_{2}\cdot e_{2}=k_{1}e_{2},\\
 &&\a(e_{1})=e_{1}+e_{2}, \a(e_{2})=e_{2}, \b(e_{1})=k_{1}e_{1}+k_{2}e_{2}, \b(e_{2})=k_{1}e_{2};\\
 &(8)& e_{2}\cdot e_{1}=k_{1}e_{1}, \a(e_{1})=0, \a(e_{2})=e_{1}, \b(e_{1})=k_{1}e_{1}, \b(e_{2})=k_{2}e_{2};\\
 &(9)& e_{1}\cdot e_{1}=k_{2}e_{2}, e_{1}\cdot e_{2}=k_{1}e_{2}, e_{2}\cdot e_{1}=k_{1}e_{1}, e_{2}\cdot e_{2}=k_{2}e_{1},\\
 &&\a(e_{1})=e_{2}, \a(e_{2})=e_{1}, \b(e_{1})=k_{1}e_{1}+k_{2}e_{2}, \b(e_{2})=k_{2}e_{1}+k_{1}e_{2}; \\
 &(10)&e_{1}\cdot e_{1}=-(k_{1}+k_{2})e_{1}, e_{1}\cdot e_{2}=k_{1}e_{1}, e_{2}\cdot e_{1}=-(k_{1}+k_{2})e_{1}, e_{2}\cdot e_{2}=k_{1}e_{1},\\
 &&\a(e_{1})=e_{1}, \a(e_{2})=e_{1}, \b(e_{1})=-(k_{1}+k_{2})e_{1}, \b(e_{2})=k_{1}e_{1}+k_{2}e_{2};\\
 &(11)&e_{1}\cdot e_{1}=k_{1}e_{1}+k_{2}e_{2}, e_{1}\cdot e_{2}=k_{2}e_{1}+(k_{1}-k_{2})e_{2}, e_{2}\cdot e_{1}=k_{1}e_{1}, e_{2}\cdot e_{2}=k_{2}e_{1},\\
 &&\a(e_{1})=e_{1}+e_{2}, \a(e_{2})=e_{1}, \b(e_{1})=k_{1}e_{1}+k_{2}e_{2}, \b(e_{2})=k_{2}e_{1}+(k_{1}-k_{2})e_{2}; \\
 &(12)&e_{2}\cdot e_{1}=(k_{1}-k_{2})e_{1}, e_{2}\cdot e_{2}=k_{1}e_{1}+k_{2}e_{2},\\
 &&\a(e_{1})=0, \a(e_{2})=e_{1}+e_{2}, \b(e_{1})=(k_{2}-k_{1})e_{1}, \b(e_{2})=k_{1}e_{1}+k_{2}e_{2};\\
 &(13)& e_{1}\cdot e_{1}=(k_{1}+k_{2})e_{2}, e_{1}\cdot e_{2}=(k_{1}+k_{2})e_{2}, e_{2}\cdot e_{1}=k_{1}e_{1}+k_{2}e_{2},
 e_{2}\cdot e_{2}=k_{2}e_{1}+(k_{1}+k_{2})e_{2},\\
 &&\a(e_{1})=e_{2}, \a(e_{2})=e_{1}+e_{2}, \b(e_{1})=k_{1}e_{1}+k_{2}e_{2}, \b(e_{2})=k_{2}e_{1}+(k_{1}+k_{2})e_{2}; \\
 &(14)&e_{1}\cdot e_{1}=k_{1}e_{1}, e_{1}\cdot e_{2}=k_{2}e_{1}, e_{2}\cdot e_{1}=k_{1}e_{1}, e_{2}\cdot e_{2}=k_{2}e_{1}+k_{1}e_{2},\\
 &&\a(e_{1})=e_{1}, \a(e_{2})=e_{1}+e_{2}, \b(e_{1})=k_{1}e_{1}, \b(e_{2})=k_{2}e_{1}+k_{1}e_{2}; \\
 &(15)&e_{1}\cdot e_{1}=k_{1}e_{1}+k_{2}e_{2}, e_{1}\cdot e_{2}=k_{2}e_{1}+k_{1}e_{2}, e_{2}\cdot e_{1}=k_{1}e_{1}+k_{2}e_{2}, e_{2}\cdot e_{2}=k_{2}e_{1}+k_{1}e_{2},\\
 &&\a(e_{1})=e_{1}+e_{2}, \a(e_{2})=e_{1}+e_{2}, \b(e_{1})=k_{1}e_{1}+k_{2}e_{2}, \b(e_{2})=k_{2}e_{1}+k_{1}e_{2};\\
 &(16)&e_{1}\cdot e_{1}=e_{1}+k_{2}e_{2}, e_{1}\cdot e_{2}=k_{3}e_{2}, e_{2}\cdot e_{1}=k_{1}e_{2},\\
 &&\a(e_{1})=e_{1}, \a(e_{2})=k_{1}e_{2}, \b(e_{1})=e_{1}+k_{2}e_{2}, \b(e_{2})=k_{3}e_{2}\ \hbox{and}\ (k_{1}-1)k_{2}=0 \\
 &(17)&e_{1}\cdot e_{1}=e_{1}+(k_{2}+1)e_{2}, e_{1}\cdot e_{2}=k_{3}e_{2}, e_{2}\cdot e_{1}=k_{1}e_{2},\\
 &&\a(e_{1})=e_{1}+e_{2}, \a(e_{2})=k_{1}e_{2}, \b(e_{1})=e_{1}+k_{2}e_{2}, \b(e_{2})=k_{2}e_{2}\ \hbox{and}\  1+k_{3}+(k_{1}-1)k_{2}=0 \\
 &(18)& e_{1}\cdot e_{1}=e_{1}, e_{1}\cdot e_{2}=k_{1}e_{2}, \a(e_{1})=e_{1}, \a(e_{2})=0, \b(e_{1})=e_{1}, \b(e_{2})=k_{1}e_{2}; \\
 &(19)& e_{1}\cdot e_{1}=e_{1}+(k_{1}+1)e_{2}, e_{1}\cdot e_{2}=(k_{1}+1)e_{2}, \\
 &&\a(e_{1})=e_{1}+e_{2}, \a(e_{2})=0, \b(e_{1})=e_{1}+k_{1}e_{2}, \b(e_{2})=(k_{1}+1)e_{2}; \\
 &(20)& e_{1}\cdot e_{1}=k_{1}k_{3}e_{2}, \a(e_{1})=k_{1}e_{1}+k_{2}e_{2}, \a(e_{2})=k{_{1}}^{2}e_{2}, \b(e_{1})=k_{3}e_{1}+k_{4}e_{2}, \b(e_{2})=k{_{3}}^{2}e_{2}\\
 &&\ \hbox{and}\ k_{2}(k_{3}-k{_{3}}^{2})+k_{4}(k{_{1}}^{2}-k_{1})=0.
 \end{eqnarray*}
 \item {\bf Case II: The BiHom-associative product is trivial.}
 \begin{eqnarray*}
 &(21)& [e_{1},e_{1}]=(-k_{1}+k_{3})e_{2}, [e_{1},e_{2}]=k_{4}e_{2}, [e_{2},e_{1}]=-k_{2}e_{2},\\
 &&\hspace{-5mm}\a(e_{1})=e_{1}+k_{1}e_{2}, \a(e_{2})=k_{2}e_{2}, \b(e_{1})=e_{1}+k_{3}e_{2}, \b(e_{2})=k_{4}e_{2}\ \hbox{and}\ k_{1}(k_{3}-k_{4})+k_{3}(k_{2}-1)=0; \\
 &(22)& [e_{1},e_{1}]=(k_{4}-k_{1}k_{3})e_{2}, [e_{2},e_{1}]=-k_{2}k_{3}e_{2},\\
 &&\a(e_{1})=e_{1}+k_{1}e_{2}, \a(e_{2})=k_{2}e_{2}, \b(e_{1})=k_{3}e_{1}+k_{4}e_{2}, \b(e_{2})=0\  \hbox{and}\ k_{1}k_{3}=k_{4}(1-k_{2}); \\
 &(23)&[e_{1},e_{1}]=(k_{1}k_{3}-k_{2})e_{2}, [e_{1},e_{2}]=k_{1}k_{4}e_{2}, \\ &&\a(e_{1})=k_{1}e_{1}+k_{2}e_{2}, \a(e_{2})=0, \b(e_{1})=e_{1}+k_{3}e_{2}, \b(e_{2})=k_{4}e_{2}\  \hbox{and}\ k_{2}(1-k_{4})=k_{3}k_{1}\\
 &(24)&[e_{1},e_{1}]=k_{1}e_{2}, \a(e_{1})=e_{1}+k_{2}e_{2}, \a(e_{2})=e_{2}, \b(e_{1})=k_{1}e_{2}, \b(e_{2})=0;\\
 &(25)&[e_{1},e_{1}]=-k_{1}e_{2}, \a(e_{1})=k_{1}e_{2}, \a(e_{2})=0, \b(e_{1})=e_{1}+k_{2}e_{2}, \b(e_{2})=e_{2}.
 \end{eqnarray*}
 \item {\bf Case III}
 \begin{eqnarray*}
 &\hspace{-45mm}(26)&e_{1}\cdot e_{1}=e_{2}, [e_{1},e_{1}]=(k_{1}-k_{2})e_{2}, [e_{1},e_{2}]=e_{2},\\
 &&\a(e_{1})=e_{1}+k_{2}e_{2}, \a(e_{2})=e_{2}, \b(e_{1})=e_{1}+k_{1}e_{2}, \b(e_{2})=e_{2},
 \end{eqnarray*}
 where $k_{i}, i=1,2,3,4$ are parameters in $\mathrm{K}$.
 \end{itemize}
 \end{ex}

 \section*{Acknowledgment} This work is supported by Natural Science Foundation of Henan Province (No. 212300410365).

 \end{document}